\documentclass[preprint,12pt]{elsarticle}

\usepackage{amssymb}
\usepackage{float} % fix table inside context 
\usepackage{fullpage}
\usepackage{amsfonts}
\usepackage{amsmath}
\usepackage{amssymb}
\usepackage{amsbsy}
\usepackage{amsthm}
\usepackage{mathtools}
\usepackage{epsfig}
\usepackage{graphicx}
\usepackage{colordvi}
\usepackage{graphics}
\usepackage{color}
\usepackage{bm}
\usepackage{verbatim} 
\usepackage{mathrsfs} 
\usepackage{nicefrac}

\newtheorem{theorem}{Theorem}[section]
\newtheorem{thm}{Theorem}[section]

\newtheorem{lemma}{Lemma}[section]

\newtheorem{algorithm}[thm]{Algorithm}

\newtheorem{remark}[thm]{Remark}

\usepackage{algpseudocode}

\newcommand{\vertiii}[1]{{\left\vert\kern-0.25ex\left\vert\kern-0.25ex\left\vert #1
		\right\vert\kern-0.25ex\right\vert\kern-0.25ex\right\vert}}

\newcommand{\aIPh}[2]{a_h^{IP}\left( #1 , #2 \right)}

\newcommand{\normiii}[1]{{\left\vert\kern-0.25ex\left\vert\kern-0.25ex\left\vert #1
		\right\vert\kern-0.25ex\right\vert\kern-0.25ex\right\vert}}

\usepackage[llbracket,rrbracket,llparenthesis,rrparenthesis]{stmaryrd}

\journal{Journal of Computational and Applied Mathematics}

\begin{document}

\begin{frontmatter}



\title{Analysis of continuous data assimilation with large (or even infinite) nudging parameters}

\author{Amanda E. Diegel\corref{label1}}
\author{ Xuejian Li\fnref{label2}}
\author{ Leo G. Rebholz\fnref{label3}}
\cortext[label1]{Department of Mathematics and Statistics, Mississippi State University, MS, 39762; partially supported by NSF grant DMS 2110768.}
\fntext[label2]{School of Mathematical and Statistical Sciences, Clemson University, Clemson, SC, 29364, xuejial@clemson.edu; partially supported by NSF grant DMS 2152623.}
\fntext[label3]{School of Mathematical and Statistical Sciences, Clemson University, Clemson, SC, 29364, rebholz@clemson.edu; partially supported by NSF grant DMS 2152623.}


\begin{abstract}
This paper considers continuous data assimilation (CDA) in partial differential equation (PDE) discretizations where nudging parameters can be taken arbitrarily large.  We prove that { solutions are long-time optimally accurate}  for such parameters for the heat and Navier-Stokes equations (using implicit time stepping methods), with error bounds that do not grow as the nudging parameter gets large.  Existing theoretical results either prove optimal accuracy but with the error scaled by the nudging parameter, or suboptimal accuracy that is independent of it.  The key idea to the improved analysis is to decompose the error based on a weighted inner product that incorporates the (symmetric by construction) nudging term, and prove that the projection error from this weighted inner product is optimal and independent of the nudging parameter.  We apply the idea to BDF2 - finite element discretizations of the heat equation and Navier-Stokes equations to show that with CDA, they will admit optimal long-time accurate solutions independent of the nudging parameter, for nudging parameters large enough.  Several numerical tests are given for the heat equation, fluid transport equation, Navier-Stokes, and Cahn-Hilliard that illustrate the theory.
\end{abstract}


%\item Proof that nudging parameters can be taken arbitrarily large for continuous data assimilation (CDA) in partial differential equation (PDE) discretizations.
%\item Error is decomposed based on a weighted inner product that incorporates the nudging term.
%\item The projection error from this weighted inner product is proved to be optimal and independent of the nudging parameter.


\end{frontmatter}


	\section{Introduction}
	Data assimilation is a broad topic that generally refers to incorporating observations into a given dynamical system to better estimate some relevant quantities of interest. One common use of data assimilation is to generate an accurate initial condition for the dynamical system so that its future state can be accurately predicted. Data assimilation was originally made popular in the weather forecasting community \cite{Hoke_Anthes_1976_MWR,Anthes_1974_JAS,Lewis_Lakshmivarahan_2008}, and various categories of data assimilation methods exist, such as Kalman filtering and particle filtering \cite{Kalman_1960_JBE,G2009,AMGC2002,FK2018}, 3D/4D variational methods \cite{C1994,C1998,R2007,LGHL2023}, and nudging methods \cite{zou_navon_ledimet_1992,DJ2005,AN2012}. In the past decade, the continuous data assimilation (CDA) method developed by Azouani, Olson, and Titi \cite{AOT14}  has rapidly become popular due to its robustness, effectiveness, and strong mathematical foundation. 
	
	CDA works by adding a nudging term (also called feedback control or penalty term) that assimilates observations into  { dynamical} systems and attempts to drive the assimilated state to the true state. %An interpolation operator in the nudging term plays a fundamental difference to the nudging data assimilation method. 
	In general, CDA applied to a time dependent PDE 
	\begin{align*}
	u_t +F(u) & =f, \\
	u(0) = u_0,
	\end{align*}
	is written as (suppressing boundary conditions for the moment)
	\begin{align}\label{C1}
		\frac{\partial w}{\partial t}+F(w)+\mu (I_Hw-I_Hu) & =f, \\
		w(0)=0.
	\end{align} 
Here, $\mu (I_Hw-I_Hu)$ is the nudging term, $u$ is the partially observed true solution, $I_H$ is an interpolation or observation operator associated with the measurement location with $H$ representing the sparsity of observations, 
$\mu$ is the nudging parameter emphasizing the accuracy of observations, and $F$ is an operator specific to the PDE at hand. We note that an initial condition of $0$ is chosen for $w$, although long-time accuracy is typically achieved in { many} CDA analyses considering arbitrary initial conditions. Of course, the more accurate the initial condition, the quicker the convergence, in general.  While CDA is similar to traditional nudging methods, a fundamental difference is the use of the interpolation operator which is a minor modification but has a critical impact. Researchers have proved in a variety of settings and systems that by using sparse observations and appropriate $I_H$, applying CDA leads to an assimilated solution that converges to the true solution (up to optimal discretization error) exponentially fast in time, independent of the initial condition (see e.g.\cite{CGJP22,CHL20,DR22,FJT15,Bessaih_Olson_Titi_2015,Albanez_Nussenzveig_Lopes_Titi_2016,Biswas_Hudson_Larios_Pei_2017,Jolly_Martinez_Olson_Titi_2018_blurred_SQG} and references therein).  

In addition to the strong mathematical foundation, another reason CDA is attractive is its straightforward implementation. The CDA computation is carried out the same way as solving the associated evolutionary PDE, only with extra attention to the nudging term that can be easily handled in most cases.  Studies have been performed using the CDA implementation based on different interpolants $I_H$ and discretization schemes that make implementation simple (e.g. only altering the diagonal of the system matrix at each time step) \cite{RZ21,ZRSI19,Larios_Victor_2019,IMT20,GN20,GNT18}, and maintain the expected exponential convergence of discrete CDA to an optimally accurate solution.

The goal of this paper is to prove optimal long-time convergence of discrete CDA with no upper restriction on the nudging parameter $\mu$ (i.e. no scaling of the error with $\mu$).  There have been several studies that have shown excellent results with CDA with nudging parameter of 1000 and larger applied to Navier-Stokes and Cahn-Hilliard solvers where implicit time stepping is used \cite{RZ21,GLRVZ21_CDA,DR22,HRV24}; some of these works also test varying $\mu$ and show no deterioration of accuracy as $\mu$ gets larger.  However, to date, no theoretical results exist to explain these good results; current known results for CDA that are optimal either depend on $\mu$ or have an upper bound restriction on $\mu$.  Allowing $\mu$ to be very large can bring convenience of CDA implementation by allowing for a direct Dirichlet enforcement assimilation. Very large $\mu$ is numerically equivalent in the linear systems to a direct enforcement of $w(x_i)=u(x_i)$ at all measurement nodes $x_i$ associated with the determination of $I_Hu$ and $I_Hw$. Such an implementation has been mentioned in \cite {L2023}, but without any proof of accuracy. In this work, we provide a rigorous analysis to support this result in the context of CDA for evolutionary PDEs. We first propose projection operators tailored for use with CDA, and prove an optimal estimate for its approximation error, independent of the nudging parameter $\mu$.  With these new projections, we are able to prove that CDA with arbitrarily large parameters find optimal long-time accurate solutions with respect to the $L^2-$norm with error bounds independent of $\mu$ for BDF2 - finite element discretizations of the heat equation and incompressible Navier-Stokes equations.  We note that the tools and techniques developed and used for the heat and Navier-Stokes equations are extendable to many dissipative PDEs.

The paper is organized as follows. In Section \ref{P}, we introduce notation and provide mathematical preliminaries for a smooth analysis to follow. In Section \ref{C}, we state the new CDA Poisson and CDA Stokes projections and analyze their properties. In Section \ref{H}, we present the CDA convergence analysis for the discretized heat equation. In Section \ref{N}, we focus on a CDA convergence analysis of discretized Navier-Stokes equations. In Section \ref{sec:num-exp}, we provide several numerical tests that support the theory. Finally, we draw conclusions in Section \ref{Con}.

	\section{Preliminaries}\label{P}

Throughout this paper, $C$ will denote a generic constant, possibly changing at each instance, that is independent of mesh size $h$, time step $\Delta t$, nudging parameter $\mu$, and problem data. We consider $\Omega$ % \subset \mathbb{R}^d$, with $d=2,3$,
	 to be a generic polygonal or polyhedral domain with smooth boundary $\partial \Omega$. Let $H^{k}(\Omega)$ denote the Sobolev space $W^{k,2}(\Omega)$ and  $L^2(\Omega)=W^{0,2}(\Omega)$. We use $(\cdot,\cdot)$ to denote the $L^2$ inner product that induces the $L^2$ norm $\|\cdot\|$,  while all other inner products or norms will be appropriately subscripted. Define the following function spaces:
	\begin{align*}
		& H^1_0\left(\Omega\right):=
		\{v\in H^1\left(\Omega\right): v=0~~\text{on}~ \partial\Omega\},\\
		& X:= \left( H^1_0\left(\Omega\right)\right)^d, \quad d=2,3,\\
	%	\{v\in H^1\left(\Omega\right): v=0~~\text{on}~ \partial\Omega\},\\
		&Q:=\{v\in {L}^2(\Omega): \int_{\Omega}vdx=0\},\\
			& V:=
		\{v\in X:  (\nabla \cdot v,q)=0\ \forall q\in Q\}.
	\end{align*}
%Denote $\langle \cdot,\cdot\rangle$ as a generic duality pair between a Hilbert space and its dual space, such as the duality pair between $H^1_0(\Omega)$ and $H^{-1}$ which is equipped with dual norm $\|\cdot\|_{-1}$.

Recall that the Poincar\'{e} inequality holds on $X$: there exists a constant $C_P$ dependent only on the domain $\Omega$ satisfying 
\[
\| v \| \le C_P \|\nabla v \|,
\]
 for all $v\in X$.

\subsection{Finite element preliminaries}
We consider finite element (FE) methods for the spatial discretization of the CDA algorithm. Denote by $\tau_h(\Omega)$ a shape-regular and conforming triangulation of the interested domain $\Omega$. We use $Y_h=H^1_0\left(\Omega\right)
\cap P_k(\tau_h(\Omega))$ as the FE discretization space for the heat equation. Let $X_h\times Q_h\subset X\times Q$ 
%\cap P_{k-1}(\tau_h(\Omega))$ 
be a pair of inf-sup stable FE spaces for Navier-Stokes equation, such as $X_h=X\cap P_{k}(\tau_h(\Omega))$ and $Q_h=Q\cap P_{k-1}(\tau_h(\Omega))$ as Taylor-Hood or Scott-Vogelius elements.

\subsection{BDF2 preliminaries}

The time stepping method used in our analysis and most of our numerical tests is BDF2, and we give here some preliminary results for BDF2 analysis.  We will utitlize the G-norm \cite{HairerWanner}: $\|x\|^2_G:=(x,Gx)_{\mathbb{R}^2}$, where $G=\begin{bmatrix}\frac{1}{2}&-1\\-1&\frac{5}{2}\end{bmatrix}$. It is easy to show that for $x\in (L^2(\Omega))^2$, the G-norm is equivalent to the $(L^2(\Omega))^2$ norm:
\begin{align}\label{G-norm}
	C_{\ell} \|x\|_G\leq \|x\|\leq C_u\|x\|_G.
\end{align}
Here, $C_{\ell}=3-2\sqrt{2}$ and $C_u=3+2\sqrt{2}$. For $v^i\in L^2, i=n-1,n,n+1$, we have the identity \cite{HairerWanner}:
\begin{align}\label{BDF2i}
	\begin{split}
	\left(\frac{1}{2}\left(3v^{n+1}-4v^n+v^{n-1}\right),v^{n+1}\right)&=\frac{1}{2}\left(\|[v^{n+1};v^n]\|_{{G}}^2-\|[v^{n};v^{n-1}]\|_{{G}}^2\right)\\&+\frac{1}{4}\|v^{n+1}-2v^n+v^{n-1}\|^2.
	\end{split}
\end{align}

\subsection{CDA preliminaries}

Denote by 
$\tau_H(\Omega)$ a coarse mesh of $\Omega$, whose nodes are locations of the measurements or observables of the true
solution. Assume $h<<H$ and the nodes of mesh $\tau_H(\Omega)$ are a subset of the vertices of the mesh $\tau_h(\Omega)$
{ and that}
%Important properties of $I_H$ are that it be determined by the measurement nodes $x_i$ (i.e. from the nodes of $\tau_H$); that is, 
if $I_h$ is the nodal interpolant operator on $\tau_h$, then 
$I_H I_h u = I_H u$ for any $u\in H^1(\Omega)$.
Also, $I_H$ must satisfy the
usual CDA estimates: there exists a constant $C_I$ independent of $H$ satisfying
\begin{align}
	\label{interpolationi}	\|I_Hv-v\|\leq C_IH \| \nabla v\|,~~~~\forall v\in H^1(\Omega),\\
	\label{interpolationi2}	\|I_Hv\|\leq C_I \|v\|,~~~~\forall v\in H^1(\Omega). 
\end{align}
We note that if $I_H$ were the nodal interpolant, then $C_I$ in \eqref{interpolationi} { could be large due to an inverse dependence on $h$ for $v\in X_h$ or $Y_h$ and so this is not an ideal interpolant choice in most cases}.  Instead, algebraic nudging \cite{RZ21}, the Scott-Zhang interpolant, or the $L^2$ projection onto the coarse space could be used \cite{GN20}.

We also need the following lemma for the long-time analysis that is proven in \cite{LLC2019}.
\begin{lemma}\label{exp}
	Suppose constants $r>1$ and $B\geq 0$. If a real number sequence $\{a_n\}_{n\geq 0}$ satisfies
	\begin{align}
		ra_{n+1}\leq a_n+B,
	\end{align}
	then 
	\begin{align}
		a_{n+1}\leq a_0\left(\frac{1}{r}\right)^{n+1}+\frac{B}{r-1}.
	\end{align}
\end{lemma}

\section{CDA Poisson and CDA Stokes projections}\label{C}

In this section we will present and analyze Poisson and Stokes interpolants tailored to use with CDA.  The estimates derived in this section are the keys that allow for the optimal long-time accuracy results that are independent of the nudging parameter to hold in the sections that follow.

\subsection{CDA Poisson projection}
The CDA Poisson projection is defined as follows: Given $u\in H^1_0(\Omega)$ and $I_H u$, find $u_h:= P_hu\mapsto u_h$ in $Y_h$ satisfying
\begin{align}\label{NPP}
	\kappa (\nabla u_h,\nabla v_h ) + \mu(I_H u_h, I_H v_h) = \kappa (\nabla u,\nabla v_h) + \mu(I_H u,I_H v_h), \quad \forall v_h \in Y_h,
\end{align}
where $\kappa > 0$ is a positive constant.

\begin{lemma}\label{cdapoisson}
	Given $u\in H^1_0(\Omega)\cap H^{k+1}(\Omega)$ and $\mu\ge 0$,  $P_h u$ satisfies
\begin{align}
	\| \nabla (u - P_h u) \| &\le Ch^k \| u \|_{H^{k+1}(\Omega)},\\
	\| u - P_h u \| &\le Ch^{k+1} \| u \|_{H^{k+1}(\Omega)},
\end{align}
	where $C$ is independent of $u$, $h$, and $\mu$.

\end{lemma}
\begin{proof}

We first rewrite the equation (\ref{NPP}) as follows:
	\begin{equation}
		(\nabla (u - u_h),\nabla v_h) + \frac{\mu}{\kappa} (I_H (u-u_h),I_H v_h)=0, \quad \forall v_h\in Y_h. \label{L2}
	\end{equation}
%Note that, the $u_h$ in (\ref{L2}) satisfies (\ref{NPP})  and functions as $P_hu$ here.
	Set $e=u- u_h = (u - I_h u) + (I_h u - u_h) =: \eta + \phi$, where $I_h u \in Y_h$ is the nodal interpolant of $u$.  Properties of $I_h$ and $I_H$ imply that 
	\[
	I_H \eta = I_H ( u - I_h u ) = I_Hu - I_HI_hu = I_H u - I_H u = 0.
	\]  
	Using this, after choosing $v_h = \phi$ in \eqref{L2}, then yields
	\begin{equation}\label{break}
	 \| \nabla \phi \|^2 + \frac{\mu}{\kappa} \| I_H \phi \|^2 =  (\nabla \eta,\nabla \phi) \le \frac{1}{2} \| \nabla \phi\|^2 + \frac{1}{2} \| \nabla \eta \|^2,
	\end{equation}
	thanks to Young's inequality.  Reducing (\ref{break}) gives
	\begin{equation}
		 \| \nabla \phi \|^2 + \frac{2\mu}{\kappa} \| I_H \phi \|^2 \le \| \nabla \eta \|^2.
	\end{equation}
With the triangle inequality and standard approximation theory, we get
	\begin{equation}
		\| \nabla e \|^2 + \frac{\mu}{\kappa} \| I_H e \|^2 \le 4 \| \nabla (u - I_h u) \|^2 \le Ch^{2k} \| u \|_{H^{k+1}(\Omega)}^2, \label{L4}
	\end{equation}
	where $C$ is independent of $h$ and $\mu$.  This proves the first result of the lemma.
	
	For the $L^2$ estimate, we begin by defining $\hat \xi \in C^2(\Omega)\cap C^0(\overline{\Omega})$ satisfying
	\[
	-\Delta \hat \xi = e, \quad\mbox{ on } \Omega.
	\]
	Let $\psi$ be any harmonic function { (e.g.~a harmonic polynomial \cite{ABR01})} that satisfies $\psi(x_i)=\hat \xi(x_i)$ for all the measurement points $x_i$.  Then we have for $\xi:=\hat \xi - \psi$ that
	\[
	-\Delta \xi = e, \ I_H \xi=0.
	\]
	Testing the first $\xi$ equation with $e$ yields
	\begin{equation}\label{L5b}
	\| e \|^2 = (\nabla \xi,\nabla e).
	\end{equation}
	From \eqref{L2} with $v_h=I_h \xi$, we have that
	\[
	(\nabla e,\nabla I_h \xi) + \mu (I_H e,I_H I_h \xi)=0,
	\]
	and so combining this with \eqref{L5b} gives
	\[
	\| e \|^2 = (\nabla e,\nabla (\xi - I_h \xi)) - \mu (I_H e,I_H I_h \xi).
	\]
	Since the measurement nodes $\{ x_i \}_{i=1}^M$ are nodes on $\tau_h$ and $\tau_H \subset \tau_h$, $I_H I_h \xi = I_H \xi=0$.  Hence we have the estimate
	\begin{align*}
	\| e\|  & =  (\nabla e,\nabla (\xi - I_h \xi)) \\
	& \le \| \nabla e \| \| \nabla (\xi - I_h \xi ) \| \\
	& \le Ch \| \nabla e \| \| \Delta  \xi  \| \\
	& = Ch \| \nabla e \| \| e\|,
	\end{align*}
	where $C$ is an interpolation constant and independent of $\mu$.  This reduces
	\[
	\| e \| \le Ch \| \nabla e\|,
	\]
	and using the bound \eqref{L4} for $\| \nabla e \|$ proves the second result of the lemma.

\end{proof}
%\begin{remark}
%Besides the optimal error order, one other purpose in the above proof is to prevent $\mu$ appearing in the estimate. To achieve this, we used two techniques: nodal interpolant $I_h$ and the harmonic approximation. The nodal interpolant $I_h$ enables $\mu(I_HI_h\eta-I_Hu)=0$ to obtain $H^1$ estimate and the harmonic approximation is used to get $L^2$ estimate.
%\end{remark}

\subsection{CDA Stokes projection}

We now consider the CDA Stokes projection, which is defined as follows:  Given $u\in X$ and $I_H u$,
find $u_h:= P_h^Su\mapsto u_h$ in $V_h=\{v\in X_h: (\nabla\cdot v,q)=0~~\forall q\in Q_h \}$ satisfying
\begin{equation}
	\begin{split}
\nu (\nabla u_h,\nabla v_h) + \mu(I_H u_h,I_H v_h)&=\nu (\nabla u,\nabla v)+\mu(I_Hu,I_Hv_h), \quad \forall v_h\in V_h. \label{L1S}
	\end{split}
\end{equation}

At first, this projection may appear to be very similar to the CDA Poisson projection with $Y_h$ replaced by $V_h$. However, there is a subtle difference regarding the fact that $I_H$ does not necessarily map into the (discretely divergence free space) $V_h$.  Hence, there is a more delicate analysis needed to prove that this projection is optimal in $L^2$ and $H^1$.  

For simplicity, we will consider Scott-Vogelius elements with $X_h = P_k(\tau_h)\cap X$ and $Q_h = P_{k-1} \cap L^2_0(\Omega)$ with $k=2$, but the same idea can be applied for $k>2$ in 2D and for $k\ge 3$ in 3D (for lower order elements than these, the following construction would not work).  The key idea is to alter the nodal interpolation $I_h$ similar to what was done in \cite{X2016} as follows: $I_h^S$ will be a nodal interpolant at all the vertices of $\tau_h$ (and therefore also at all the measurement nodes so that $I_H I_h^S = I_H$), but the value at the midedges is defined so that the total flux through any element boundary is 0.  Hence on each edge $e$, where $\{x_i\}_{i=1}^{3}$ are the nodes on $e$ (i=1 and 3 correspond to vertices, i=2 to the midedge), and $\vec{n}$ is the unit outward normal, $\psi_i(x)$ is the finite element nodal basis that belongs to support point $x_i$, $\vec{e}_j$ is a $d$ dimension unit vector with $j^{th}$ component $1$ and others $0$, ${\psi}_i^j(x)$ is a $d$ dimension vector with $j^{th}$ component as $\psi_i(x)$ and others $0$, we make the following correction to the midedge value:
\begin{align}
	I_h^Su(x_2)& =
		u(x_2) + \frac{\int_eu(s)\cdot \vec{n}ds-\sum_{i,j}\int_e((\vec{e}_j\cdot u(x_i))\psi_i^j(s))\cdot \vec{n}ds}{\sum_{j=1}^d\int_e\psi_2^j(s)\cdot \vec{n}ds}\vec{n}.
\end{align}
With this definition of $I_h^S$, one can calculate that the total flux through each element $E$ is zero, and thus by the divergence theorem that $\int_{E} \nabla \cdot I_h^S u\ dx=0$.  With this, since $\nabla \cdot X_h\subset Q_h$, we get that $\| \nabla \cdot I_h^S u \|_{L^2(E)}=0$ and thus $\| \nabla \cdot I_h^S u \|=0$.  Therefore, it holds that $I_h^S: H^1\rightarrow V_h$.

In addition to being divergence-free, it is proven in \cite{X2016} that $I_h^S$ has the following property.

\begin{lemma}\label{lemma1}
	For $u\in L^{\infty}(\Omega)$, the interpolant $I_h^S$ satisfies
	\begin{align}\label{Ih}
		\max_{x\in e}|u(x)-I_h^Su(x)|\leq Ch^{k+1}.
	\end{align}
\end{lemma}  
Since all nodes of $\tau_h$ are on edges of the mesh, we immediately get that the max error at any node is $Ch^{k+1}$.  This can be used with standard approximation theory to get that
\begin{align}
	\label{NI1}\|u(x)-I_h^S(u)\|&\leq Ch^{k+1} \| u \|_{H^{k+1}(\Omega)},\\
	\label{NI2}	\|\nabla \left(u-I_h^Su\right)\|&\leq Ch^{k} \| u \|_{H^{k+1}(\Omega)}.
\end{align}

With these approximation properties and with $I_h^S$ mapping into $V_h$, we have the following result for the CDA-Stokes projection:
\begin{lemma}\label{cdastokes}
	Given $u\in X\cap H^{k+1}(\Omega)$ and $\mu\ge 0$,  the Stokes projection $P_h^S u$ satisfies
\begin{align}
	\| \nabla (u - P^S_h u) \|& \le Ch^k \| u \|_{H^{k+1}(\Omega)},\\
	\| u - P^S_h u \| &\le Ch^{k+1} \| u \|_{H^{k+1}(\Omega)},
\end{align}
	where $C$ is independent of $u$, $h$, and $\mu$.
\end{lemma}

\begin{proof}
Follow the proof of Lemma \ref{cdapoisson} with $I_h$ replaced by $I_h^S$ and $Y_h$ replaced by $V_h$.
\end{proof}

	\section{Analysis of large parameter CDA for BDF2-FEM discretization of heat equation}\label{H}
	
We now consider an analysis of CDA with large nudging parameters applied to the BDF2-FEM discretization of the heat equation with form
	\begin{align}\label{heatw}
	\begin{split}
	u_t - \kappa \Delta u & = f,\\
	 u(\cdot,0)&=u_0, \\
	 u|_{\partial\Omega} & = 0,
 \end{split}
\end{align}
where $\kappa > 0$ is the thermal diffusivity constant.
	
Introducing the bilinear form $a(\cdot,\cdot): H^1_0\times H^1_0\rightarrow \mathbb{R}$ defined by $a(\phi,v)=(\kappa\nabla \phi,\nabla v)$, the CDA algorithm we consider herein for the heat equation using a BDF2-FEM discretization is as follows:
\begin{algorithm}\label{alg:heat}
Let $\Delta t$ be an uniform time step and $Y_h$ be the FE space. Given any initials condition $w_h^0\in L^2(\Omega), w_h^1\in Y_h$, source $f\in L^{\infty}(0,\infty; L^2(\Omega))$, $I_H$ satisfying (\ref{interpolationi})-(\ref{interpolationi2}),  observations $\{I_Hu^{n+1}\}_{n\geq 1}\in L^2(\Omega)$, and nudging parameter $\mu>0$, find $\{w^n_h\}\in Y_h$ for $n=1,2,3,\cdots,$ satisfying 
\begin{align}\label{DCDA}
\left(\frac{3w_h^{n+1}-4w_h^{n}+w_h^{n-1}}{2\Delta t}, v_h\right)+ a(w_h^{n+1},v_h)+\mu  \left(I_H(w_h^{n+1}-u^{n+1}), I_Hv_h\right)=(f,v_h),
\end{align}
for all $v_h\in Y_h$.
\end{algorithm}
%Consider the weak formulation of Heat equation: Find $u\in H^{1}_{0}(\Omega)$ satisfying
%\begin{align}\label{HC}
	%\begin{split}
%	\langle \frac{\partial u}{\partial t} ,v\rangle +a(u,v)&=\langle f ,v\rangle \quad \forall v\in H^{1}_{0}(\Omega),\\
%	u(0)&=u_0  \quad  \text{in~~} L^2(\Omega).
%	\end{split}
%\end{align}
%Based on (\ref{heatw}), 

We now prove that this method has temporally and spatially optimal long-time accuracy with respect to the $L^2-$norm for any $\mu>0$.

\begin{theorem}\label{thm1}
	Let $u\in L^{\infty}(0,\infty;H^{k+1}(\Omega))$ be the solution of the heat equation (\ref{heatw}) where $u_t, u_{tt}, u_{ttt}\in L^{\infty}(0,\infty;L^2(\Omega))$, $u^0=u_0\in L^2(\Omega),\ u^1\in H_{{0}}^1(\Omega)$, and $f\in L^{\infty}(0,\infty;L^2(\Omega))$. Then, for any given $H>0$ and $\mu\geq \frac{1}{2C_I^2H^2}$, the difference between the CDA solution of (\ref{DCDA}) and the true solution satisfies: 
\begin{align}
	\begin{split}
		\|w_h^{n+1}-u^{n+1}\|^2&\leq C\left(\frac{1}{1+\lambda\Delta t}\right)^{n+1}\left(\|[w_h^{1}-u^{1};w_h^{0}-u^{0}]\|^2_G+\frac{\kappa\Delta t}{4}\|\nabla (w_h^{1}-u^{1})\|^2\right)\\&~~~~ +\frac{C}{\lambda}\left(\Delta t^4+h^{2k+2}\right)
	\end{split}
\end{align}
for any $n\geq 1$. Here, $\lambda=\min\{\Delta t^{-1}, \frac{\kappa C_{\ell}^2 }{2C_I^2H^2},\frac{\kappa C_{\ell}^2 }{4C_P^{2}}\}$ and the constant $C$ is independent of $\mu$, $h$, and $\Delta t$.
\end{theorem}

\begin{remark}
	In (\ref{thm1}), the nudging parameter $\mu$ does not appear on any right side terms, which implies that there will be no adverse effect by considering arbitrarily large $\mu$ in the CDA algorithm. 
\end{remark}

\begin{proof}
	Using Taylor's theorem, the solution of the heat equation (\ref{heatw}) satisfies
	\begin{align}\label{DCH}
		\left(\frac{3u^{n+1}-4u^{n}+u^{n-1}}{2\Delta t}, v_h\right)+a(u^{n+1},v_h)&=(f,v_h)+\frac{\Delta t^2}{3}(u_{ttt}(t^*),v_h),~~\forall v_h\in Y_h,
	\end{align}
where $t^*\in[t^{n-1},t^{n+1}]$.	%for a given $u^0=u_0\in L^2(\Omega)$.
	Subtracting equation (\ref{DCDA}) from (\ref{DCH}) gives the error equation
\begin{align}\label{ee}
	\begin{split}
	\left(\frac{3e^{n+1}-4e^{n}+e^{n-1}}{2\Delta t}, v_h\right)	&+a(e^{n+1},v_h)+\mu \left(I_He^{n+1}, I_Hv_h\right)\\&=\frac{\Delta t^2}{3}(u_{ttt}(t^*),v_h),~~\forall v_h\in Y_h,	
	\end{split}
\end{align}
where $e^n=u^n-w_h^n$.
We decompose $e^n$ as: $e^n=\left(u^n-P_hu^n\right)+\left(P_hu^n-w_h^n\right)=\eta^n+\phi^n$. Then with the projection equation (\ref{NPP}), we have that $a(\eta,v_h)+\mu(I_H \eta,I_H v_h)=0$, so we can rewrite (\ref{ee}) as 
\begin{align}\label{ed}
\begin{split}
	&\left(\frac{3\phi^{n+1}-4\phi^{n}+\phi^{n-1}}{2\Delta t}, v_h\right)	+a(\phi^{n+1},v_h)+\mu \left(I_H\phi^{n+1}, I_Hv_h\right)\\&=-\left(\frac{3\eta^{n+1}-4\eta^{n}+\eta^{n-1}}{2\Delta t}, v_h\right)+\frac{\Delta t^2}{3}(u_{ttt}(t^*),v_h).
\end{split}
\end{align}
Taking $v_h=\phi^{n+1}$ in equation (\ref{ed}), using the BDF2 identity (\ref{BDF2i}), and multiplying by $2\Delta t$, we have %$a(a-b)=\frac{1}{2}(a^2-b^2)+\frac{1}{2}(a-b)^2$, we have 
\begin{align}\label{ed1}
	\begin{split}
&	\|[\phi^{n+1};\phi^{n}]\|^2_G-\|[\phi^{n};\phi^{n-1}]\|^2_G+\frac{1}{2}\|\phi^{n+1}-2\phi^{n}+\phi^{n-1}\|^2+2\kappa\Delta t\|\nabla \phi^{n+1}\|^2\\&+2\mu\Delta t  \|I_H\phi^{n+1}\|^2
		=-\left(3\eta^{n+1}-4\eta^{n}+\eta^{n-1}, \phi^{n+1}\right)+\frac{2\Delta t^3}{3}(u_{ttt}(t^*),\phi^{n+1}).
	\end{split}
\end{align}
Following the analysis from \cite{RZ21} for the time derivative term and using the Cauchy-Schwarz and Young's inequalities on the right-hand side terms of (\ref{ed1}), we obtain
\begin{align}\label{ed2}
	\begin{split}
	-\left(3\eta^{n+1}-4\eta^{n}+\eta^{n-1}, \phi^{n+1}\right)&\leq C\kappa^{-1}\Delta t\left(\|\eta_t\|^2_{L^{\infty(0,\infty;L^2)}}+\int_{t^{n-1}}^{t^{n+1}}\|\eta_{tt}\|^2dt\right)+\frac{\kappa\Delta t}{4}\|\nabla \phi^{n+1}\|^2,\\
	\frac{2\Delta t^3}{3}(u_{ttt}(t^*),\phi^{n+1})&\leq  C\kappa^{-1}\Delta t^5\|u_{ttt}(t^*)\|^2    +\frac{\kappa\Delta t}{4}\|\nabla \phi^{n+1}\|^2.
	\end{split}
\end{align}
Using the interpolation property (\ref{interpolationi}), we { provide the lower bound of} the left-hand side terms of equation (\ref{ed1}) by %\cite{L2023,L2024}
\begin{equation}\label{ed3}
	\begin{split}
		2 \kappa\Delta t\|\nabla \phi^{n+1}\|^2+&2\mu\Delta t\|I_H\phi^{n+1}\|^2  =\kappa\Delta t\|\nabla \phi^{n+1}\|^2+2\mu\Delta t\|I_H\phi^{n+1}\|^2+\kappa\Delta t\|\nabla \phi^{n+1}\|^2\\
		&\geq\frac{\kappa\Delta t}{C_I^2H^2}\|\phi^{n+1}-I_H\phi^{n+1}\|^2+2\mu\Delta t\|I_H\phi^{n+1}\|^2+\kappa\Delta t\|\nabla \phi^{n+1}\|^2\\
		&\geq \beta \left(\|\phi^{n+1}-I_H\phi^{n+1}\|^2+\|I_H\phi^{n+1}\|^2\right)+\kappa\Delta t\|\nabla \phi^{n+1}\|^2\\
		&\geq \frac{\beta}{2} \|\phi^{n+1}\|^2+\kappa\Delta t\|\nabla \phi^{n+1}\|^2,
	\end{split}
\end{equation}
where $\beta=\min\{2\mu\Delta t,\frac{\kappa\Delta t}{C_I^2H^2}\}= \frac{\kappa\Delta t}{C_I^2H^2}$ by allowing $\mu$ to be sufficient large, i.e., $\mu\geq\frac{\kappa }{2C_I^2H^2}$. \\
Combining (\ref{ed1}), (\ref{ed2}), and (\ref{ed3}) results in
\begin{align}\label{ed4}
	\begin{split}
		\|[\phi^{n+1};\phi^{n}]\|^2_G&+\frac{\kappa\Delta t}{4}\|\nabla \phi^{n+1}\|^2+\frac{\kappa\Delta t}{4}\|\nabla \phi^{n+1}\|^2+\frac{\kappa\Delta t}{2C_I^2H^2}\|\phi^{n+1}\|^2+\frac{\kappa\Delta t}{4}\|\nabla \phi^{n}\|^2\\&\leq 	\|[\phi^{n};\phi^{n-1}]\|^2_G+\frac{\kappa\Delta t}{4}\|\nabla \phi^{n}\|^2
	 + C\kappa^{-1}\Delta t^5\|u_{ttt}(t^*)\|^2 \\&+C\kappa^{-1}\Delta t\left(\|\eta_t\|^2_{L^{\infty}(0,\infty;L^2)}+\int_{t^{n-1}}^{t^{n+1}}\|\eta_{tt}\|^2dt\right),
	\end{split}
\end{align}
where we added term $\frac{\kappa\Delta t}{4}\|\nabla \phi^{n}\|^2$ on both sides of (\ref{ed4}). 

After using the Poincar\'{e} inequality and norm equivalence (\ref{G-norm}), we obtain
\begin{align}\label{ed44}
	\begin{split}
		&\|[\phi^{n+1};\phi^{n}]\|^2_G+\frac{\kappa\Delta t}{4}\|\nabla \phi^{n+1}\|^2+\lambda\Delta t\left(\|[\phi^{n+1};\phi^{n}]\|^2_G+\frac{\kappa\Delta t}{4}\|\nabla \phi^{n+1}\|^2\right)\leq 
		\\	
		&\quad\|[\phi^{n};\phi^{n-1}]\|^2_G +\frac{\kappa\Delta t}{4}\|\nabla \phi^{n}\|^2+ C\kappa^{-1}\Delta t^5\|u_{ttt}(t^*)\|^2 
		\\
		&\qquad+C\kappa^{-1}\Delta t\left(\|\eta_t\|^2_{L^{\infty}(0,\infty;L^2)}+\int_{t^{n-1}}^{t^{n+1}}\|\eta_{tt}\|^2dt\right).
	\end{split}
\end{align}
Here, $\lambda=\min\{\Delta t^{-1}, \frac{\kappa C_{\ell}^2 }{2C_I^2H^2},\frac{\kappa C_{\ell}^2 }{4C_P^{2}}\}$.
Using Lemma \ref{exp} and the estimate on $\eta$, we have
\begin{align*}
			\|[\phi^{n+1};\phi^{n}]\|^2_G+\frac{\kappa\Delta t}{4}\|\nabla \phi^{n+1}\|^2&\leq \left(\frac{1}{1+\lambda\Delta t}\right)^{n+1}\left(\|[\phi^{1};\phi^{0}]\|^2_G+\frac{\kappa\Delta t}{4}\|\nabla \phi^1\|^2\right) 
			\\
			&\quad+\frac{C}{\lambda}\left(\Delta t^4+h^{2k+2}\right).
\end{align*}
Finally, the proof of Theorem \ref{thm1} is completed by applying the triangle inequality and another application of the norm equivalence property (\ref{G-norm}). 
\iffalse
\begin{align}\label{ed6}
	\begin{split}
	&\|w_h^{n+1}-u^{n+1}\|^2\leq 2\left(\|\eta^{n+1}\|^2+	\|\phi^{n+1}\|^2\right)\\&\leq \left(\frac{1}{\left(1+\lambda\Delta t\right)}\right)^{n+1}\|\phi^{0}\|^2 +\frac{4C_I^2H^2}{\Delta t}\left(4C_P^2\Delta t\|\eta_t(t^{**})\|^2+C_P^2\Delta t^3\|u_{tt}(t^*)\|^2\right).
	\end{split}
\end{align}
which completes the proof.
\fi
\end{proof}

\section{Discrete CDA analysis for Navier-Stokes equations}\label{N}

We consider a BDF2-FEM algorithm with CDA for the NSE, and assume divergence-free elements $(X_h,Q_h)$ that also satisfy the discrete inf-sup condition.  We assume $u$ solves the Navier-Stokes equations:
\begin{align}\label{nnsew}
	\begin{split}
u_t + u\cdot\nabla u + \nabla p - \nu\Delta u & = f, \\
\nabla \cdot u &=0,\\
u(\cdot,0)&= u_0,\\
u|_{\partial\Omega} & = 0.
\end{split}
\end{align}
We define the following bilinear and trilinear forms, respectively:
\begin{align}
a^N(\phi,v)=( \nu\nabla \phi,\nabla v),\quad \forall \phi,v\in X,\\
b(\phi,w,v)=(  \phi\cdot\nabla w, v),\quad \forall \phi,w,v\in X.
\end{align}
%Under the inf-sup condition, the Navier-Stokes equation (\ref{nnsew}) is equivalent to 
%\begin{align}\label{nsew}
%	\begin{split}
%		\left\langle\frac{\partial u}{\partial t},v\right\rangle+a^N( u, v)+b(u,u,v)&=\langle f,v\rangle \quad \forall v\in V, \\
%		u(\cdot,0)&=u_0\quad \text{in~} L^2(\Omega).
%	\end{split}
%\end{align}
By construction, $b(\phi,v,v)=0$ for all $\phi,v\in X$ with $\| \nabla \cdot \phi\|=0$,  and additionally have the following upper bounds \cite{laytonBook}:
 for any $\phi,w,v\in X$ it holds that for some $C>0$ dependent only on the size of $\Omega$,
\begin{align}
%	\label{n11}	b(u,w,v)&\leq M \|\nabla u\|^{\frac{1}{2}}\| u\|^{\frac{1}{2}}\|\nabla w\|\|\nabla v\|,\\ 
	\label{n22}	b(\phi,w,v)&\leq C \|\nabla \phi\|\|\nabla w\|\|\nabla v\|,\\
	\label{n33}	b(\phi,w,v)& \leq C \left(\|w\|_{L^{\infty}}+\|\nabla w\|_{L^3}\right)\| \phi\|\|\nabla v\|.
\end{align}

\begin{algorithm} \label{alg1}
	Let $\Delta t>0$ be the time step size. Given initials condition $w_h^0\in L^2(\Omega), w_h^1\in X_h$, source $f\in L^{\infty}(0,\infty; L^2(\Omega))$, $I_H$ satisfying (\ref{interpolationi})-(\ref{interpolationi2}),  observations $\{I_Hu^{n+1}\}_{n\geq 1}\in L^2(\Omega)$, and nudging parameter $\mu>0$, find $\{w^n_h\}\in X_h$ for $n=1,2,3,\cdots,$ satisfying 
\begin{align}\label{ND1}
	\begin{split}
		\left(\frac{3w_h^{n+1}-4w_h^{n}+w_h^{n-1}}{2\Delta t}, v_h\right)+a^N(w_h^{n+1},v_h)&+b(2w^{n}_h-w^{n-1}_h,w^{n+1}_h,v_h)\\+\mu  \left(I_H(w_h^{n+1}-u^{n+1}), I_Hv_h\right)&=(f^{n+1},v_h),~~\forall v_h\in X_h.
	\end{split}
\end{align}
\end{algorithm}

 We now prove a long time optimal $L^2$ accuracy result for Algorithm \ref{alg1}.  This improves on the analysis in \cite{RZ21,LLC2019,GN20} by providing optimal accuracy in space and time that does not scale with $\mu$ when $\mu$ gets large.

\begin{theorem}\label{thm2}
		Let $u\in L^{\infty}(0,T;H^{k+1}(\Omega))$ solve the Navier-Stokes equations (\ref{nnsew}), \newline $u_{tt}\in L^{\infty}(0,T;H^1(\Omega))$, $u_{ttt}\in L^{\infty}(0,T;L^2(\Omega))$, $u^0 = u_0\in L^2(\Omega), u^1\in X$, and $f\in L^{\infty}(0,T;L^2(\Omega))$. Assume the time step $\Delta t$ satisfies 
		\begin{align*}
		\frac{1}{2}-C\nu^{-1}\Delta t\left(h^{2k+2}+h^{2k}+\|u^{n+1}\|^2_{L^{\infty}}+\|\nabla u^{n+1}\|^2_{L^3}\right)>0 \quad \text{and} \quad \mu\geq \frac{\nu}{4C_I^2 H^2}.
		\end{align*} If 
		\begin{align}
			\frac{\nu}{4C_I^2 H^2}-C\nu^{-1}\Delta t\left(h^{2k+2}+h^{2k}+\|u^{n+1}\|^2_{L^{\infty}}+\|\nabla u^{n+1}\|^2_{L^3}\right)>0,
		\end{align}
		then the difference between the CDA solution of (\ref{ND1}) and the true Navier-Stokes equations' solution satisfies: for any $n\geq 1$
%	Let $u\in L^{\infty}(0,T;H^{k+1}(\Omega))$ denote the true solution of the NSE (\ref{NSEC}). Assume $f\in L^2(0,T;L^2(\Omega))$ and $u_t\in L^{\infty}(0,T;L^2(\Omega))$, for $H$ satisfying $\frac{\nu}{2C_I^2H^2}-C(\nu,M)\|\nabla w^{n+1}_h\|^4>0$ and $\mu\geq \frac{1}{2C_I^2H^2}$, the error in solution of the data assimilation equation (\ref{ND}) to the true solution (solution of the NSE (\ref{NSEC}))  satisfies: for any $n\geq 0$
\begin{align}\label{NS99}
	\begin{split}
		\|w_h^{n+1}-u^{n+1}\|^2&\leq C\left(\frac{1}{1+\lambda\Delta t}\right)^{n+1}\left(\|[w_h^{1}-u^{1};w_h^{0}-u^{0}]\|^2_G+\frac{\kappa\Delta t}{4}\|\nabla (w_h^{1}-u^{1})\|^2\right)\\&~~~~+\frac{C\nu^{-1}\left(h^{2k+2}+h^{4k}+\Delta t^4\right)}{\lambda }.
	\end{split}
\end{align}
Here, $\lambda=\min\left\{\Delta t^{-1},\frac{\nu C_l^2}{4C_P^{2}},\left(\frac{\nu}{4C_I^2 H^2}-C\nu^{-1}\left(h^{2k+2}+h^{2k}+\|u^{n+1}\|^2_{L^{\infty}}+\|\nabla u^{n+1}\|^2_{L^3}\right)\right)C_l^2\right\},$ and note the constants are independent of $\mu$. 
\end{theorem}

\begin{remark}
Note that for $\Delta t$ and $H$ sufficiently small, $\lambda=\frac{\nu C_l^2}{4C_P^{2}}$.  Then for $n$ large enough, the $L^2$ accuracy is $\mathcal{O}(h^{k+1} + \Delta t^2)$.
\end{remark}

\begin{proof}
	Using Taylor's theorem, it can be shown that the true solution of the Navier-Stokes equations (\ref{nnsew}) satisfies 
	\begin{align}\label{N1}
		\begin{split}
		&\left(\frac{3u^{n+1}-4u^{n}+u^{n-1}}{2\Delta t}, v\right)	+a^N(u^{n+1},v)+b(2u^{n}-u^{n-1},u^{n+1},v)\\&=(f^{n+1},v)+\frac{\Delta t^2}{3}(u_{ttt}(t^*),v)+\Delta t^2 b(u_{tt}(t^{**}),u^{n+1},v),~~\forall v\in  X,
		\end{split}
	\end{align}
	where $t^*,t^{**}\in[t^{n-1},t^{n+1}]$.	Subtracting equation (\ref{ND1}) from (\ref{N1}) gives an error equation:
\begin{align}\label{Ne}
	\begin{split}
	&\left(\frac{3e^{n+1}-4e^{n}+e^{n-1}}{2\Delta t}, v\right)	+a^N(e^{n+1},v_h)-b(2e^{n}-e^{n-1},e^{n+1},v_h)
	\\
	&\quad+b(2u^{n}-u^{n-1},e^{n+1},v_h)+b(2e^{n}-e^{n-1},u^{n+1},v_h)+\mu \left(I_He^{n+1}, I_Hv_h\right)
	\\
	&=\frac{\Delta t^2}{3}(u_{ttt}(t^*),v_h)+\Delta t^2 b(u_{tt}(t^{**}),u^{n+1},v_h),~~\forall v_h\in X_h,
\end{split}
\end{align}
where $e^n=u^n-w^n_h$. We then decompose $e^n$ as 
\[e^n=\left(u^n-P_h^Su^n\right)+\left(P_h^Su^n-w_h^n\right)=\eta^n+\phi^n\]
and, with the help of the projection equation (\ref{L1S}), we rewrite (\ref{Ne}) as:
\begin{align}\label{NS2}
	\begin{split}
			&\left(\frac{3\phi^{n+1}-4\phi^{n}+\phi^{n-1}}{2\Delta t}, v\right)+\left(\frac{3\eta^{n+1}-4\eta^{n}+\eta^{n-1}}{2\Delta t}, v\right)	+a^N(\phi^{n+1},v_h)
			\\
			&\quad-b(2\phi^{n}-\phi^{n-1},\phi^{n+1},v_h)-b(2\phi^{n}-\phi^{n-1},\eta^{n+1},v_h)-b(2\eta^{n}-\eta^{n-1},\phi^{n+1},v_h)
			\\
			&\quad-b(2\eta^{n}-\eta^{n-1},\eta^{n+1},v_h)+b(2u^{n}-u^{n-1},\phi^{n+1},v_h)+b(2u^{n}-u^{n-1},\eta^{n+1},v_h)
			\\
			&\quad+b(2\phi^{n}-\phi^{n-1},u^{n+1},v_h)+b(2\eta^{n}-\eta^{n-1},u^{n+1},v_h)+\mu \left(I_H
			\phi^{n+1}, I_Hv_h\right)
			\\
			&=\frac{\Delta t^2}{3}(u_{ttt}(t^*),v_h)+\Delta t^2 b(u_{tt}(t^{**}),u^{n+1},v_h),~~\forall v_h\in X_h.
	%	&\left(\frac{\eta^{n+1}-\eta^{n}}{\Delta t}, v_h\right)+	\left(\frac{\phi^{n+1}-\phi^{n}}{\Delta t}, v_h\right)	+a^N(\phi^{n+1},v_h)+\mu \left(I_H\phi^{n+1}, I_Hv_h\right)+b(\phi^{n+1},\phi^{n+1},v_h)\\
%		&+b(\phi^{n+1},\eta^{n+1},v_h)+b(\eta^{n+1},\phi^{n+1},v_h)+b(\eta^{n+1},\eta^{n+1},v_h)+b(u^{n+1},\phi^{n+1},v_h)\\
%		&+b(\phi^{n+1},u^{n+1},v_h)+b(\eta^{n+1},u^{n+1},v_h)+b(u^{n+1},\eta^{n+1},v_h)
%		+(p,\nabla\cdot v_h)=\frac{\Delta t}{2}(u_{tt}(t^*),v_h).
	\end{split}
\end{align}
Taking $v_h=\phi^{n+1}$ and using the identity $b(\cdot,v,v)=0,~ \forall v\in X$, %$\forall \psi\in H^1(\Omega)$, 
we have
\begin{align}\label{NS3}
	\begin{split}
		&\frac{	1}{2\Delta t}(\|[\phi^{n+1};\phi^n]\|^2_G-\|[\phi^{n};\phi^{n-1}]\|^2_G+\frac{1}{2}\|\phi^{n+1}-2\phi^{n}+\phi^{n-1}\|^2)	+\nu\|\nabla \phi^{n+1}\|^2\\&+\mu \|I_H\phi^{n+1}\|^2
		=-\left(\frac{3\eta^{n+1}-4\eta^{n}+\eta^{n-1}}{2\Delta t}, \phi^{n+1}\right)+\Delta t^2b(u_{tt}(t^{**}),u^{n+1},\phi^{n+1})\\&
		+b(2\phi^{n}-\phi^{n-1},\eta^{n+1},\phi^{n+1})+b(2\eta^{n}-\eta^{n-1},\eta^{n+1},\phi^{n+1})-b(2u^{n}-u^{n-1},\eta^{n+1},\phi^{n+1})
		\\
		&\quad-b(2\phi^{n}-\phi^{n-1},u^{n+1},\phi^{n+1})-b(2\eta^{n}-\eta^{n-1},u^{n+1},\phi^{n+1})+\frac{\Delta t^2}{3}(u_{ttt}(t^*),\phi^{n+1}).
	\end{split}
\end{align}
Applying standard inequalities such as Young's and the Cauchy-Schwarz inequalities, properties (\ref{n22})-(\ref{n33}), we bound the right hand side term of (\ref{NS3}) as follows:
%We now focus on bounding all terms in the right hand side of equation (\ref{NS3}). First, we simply use Cauchy-Schwartz and Young's inequality to all $L^2$ inner product terms:
\begin{align}\label{NS4}
	\begin{split}
\frac{\Delta t^2}{3}(u_{ttt}(t^*),\phi^{n+1})&\leq C\nu^{-1}\Delta t^4\|u_{ttt}(t^*)\|^2+\frac{\nu}{20}\|\nabla \phi^{n+1}\|^2
\\
\Delta t^2b(u_{tt}(t^{**}),u^{n+1},\phi^{n+1})&\leq C\nu^{-1}\Delta t^4\|\nabla u^{n+1}\|^2\|\nabla u_{tt}(t^{**})\|^2+\frac{\nu}{20}\|\nabla \phi^{n+1}\|^2
\\
-\left(\frac{3\eta^{n+1}-4\eta^{n}+\eta^{n-1}}{2\Delta t}, \phi^{n+1}\right)&\leq C\nu^{-1}\left(\|\eta_t\|^2_{L^{\infty}(0,\infty;L^2)}+\int_{t^{n-1}}^{t^{n+1}}\|\eta_{tt}\|^2dt\right)+\frac{\nu}{20}\|\nabla \phi^{n+1}\|^2
\\
b(2\phi^{n}-\phi^{n-1},\eta^{n+1},\phi^{n+1})&=-b(\phi^{n+1}-2\phi^{n}+\phi^{n-1},\eta^{n+1},\phi^{n+1})+b(\phi^{n+1},\eta^{n+1},\phi^{n+1})
\\
&\leq C\nu^{-1}\left(\|\eta^{n+1}\|^2_{L^{\infty}}+\|\nabla\eta^{n+1}\|^2_{L^3}\right)\|\phi^{n+1}-2\phi^{n}+\phi^{n-1}\|^2
\\
&\quad+C\nu^{-1}\left(\|\eta^{n+1}\|^2_{L^{\infty}}+\|\nabla\eta^{n+1}\|^2\right)\|\phi^{n+1}\|^2+\frac{\nu}{10}\|\nabla \phi^{n+1}\|^2
\\
b(2\eta^{n}-\eta^{n-1},\eta^{n+1},\phi^{n+1})&\leq C\nu^{-1}\|\nabla(2\eta^{n}-\eta^{n-1})\|^2\|\nabla\eta^{n+1}\|^2+\frac{\nu}{20}\|\nabla\phi^{n+1}\|^2
\\
-b(2u^{n}-u^{n-1},\eta^{n+1},\phi^{n+1})&= b(2u^{n}-u^{n-1},\phi^{n+1},\eta^{n+1})
\\
&\leq C\nu^{-1}\|2u^{n}-u^{n-1}\|^2\|\eta^{n+1}\|^2_{L^{\infty}}+\frac{\nu}{20}\|\nabla\phi^{n+1}\|^2
\\
-b(2\phi^{n}-\phi^{n-1},u^{n+1},\phi^{n+1})&=b(\phi^{n+1}-2\phi^{n}+\phi^{n-1},u^{n+1},\phi^{n+1})-b(\phi^{n+1},u^{n+1},\phi^{n+1})
\\
&\leq C\nu^{-1}\left(\|u^{n+1}\|^2_{L^{\infty}}+\|\nabla u^{n+1}\|^2_{L^3}\right)\|\phi^{n+1}-2\phi^{n}+\phi^{n-1}\|^2
\\
&\quad+C\nu^{-1}\left(\|u^{n+1}\|^2_{L^{\infty}}+\|\nabla u^{n+1}\|^2\right)\|\phi^{n+1}\|^2+\frac{\nu}{10}\|\nabla \phi^{n+1}\|^2
\\
-b(2\eta^{n}-\eta^{n-1},u^{n+1},\phi^{n+1})&= -b(2\eta^{n}-\eta^{n-1},u^{n+1},\phi^{n+1})
\\
&\leq C\nu^{-1}\left(\|u^{n+1}\|^2_{L^{\infty}}+\|\nabla u^{n+1}\|^2_{L^3}\right)\|\eta^{n}-\eta^{n-1}\|^2
\\
&\quad+\frac{\nu}{20}\|\nabla \phi^{n+1}\|^2.
%-b(2\eta^{n}-\eta^{n-1},\eta^{n+1},\phi^{n+1})-b(2u^{n}-u^{n},\eta^{n+1},\phi^{n+1})-b(2\phi^{n}-\phi^{n-1},u^{n+1},\phi^{n+1})-b(2\eta^{n}-\eta^{n-1},u^{n+1},\phi^{n+1}).
	\end{split}
\end{align}
We use the interpolant inequality (\ref{interpolationi}) to lower bound $\nu\|\nabla \phi^{n+1}\|^2+\mu \|I_H\phi^{n+1}\|^2$, following the analysis of \cite{GN20}, as follows:
\begin{equation}\label{NS7}
	\begin{split}
		&\nu	\|\nabla \phi^{n+1}\|^2+\mu\|I_H\phi^{n+1}\|^2 \geq\frac{\nu}{4C_I^2H^2}\|\phi^{n+1}-I_H\phi^{n+1}\|^2+\mu\|I_H\phi^{n+1}\|^2+\frac{3\nu}{4}	\|\nabla \phi^{n+1}\|^2\\
		&\geq \gamma \left(\|\phi^{n+1}-I_H\phi^{n+1}\|^2+\|I_H\phi^{n+1}\|^2\right)+\frac{3\nu}{4}	\|\nabla \phi^{n+1}\|^2\geq \frac{\gamma}{2} \|\phi^{n+1}\|^2+\frac{3\nu}{4}	\|\nabla \phi^{n+1}\|^2,
	\end{split}
\end{equation}
where $\gamma=\min\{\mu,\frac{\nu}{4C_I^2H^2}\}= \frac{\nu}{4C_I^2H^2}$ by allowing $\mu$ to be sufficient large, i.e., $\mu\geq\frac{\nu}{4C_I^2H^2}$.

 Combining (\ref{NS3}), (\ref{NS4}), and (\ref{NS7}), multiplying by $2 \Delta t$, and taking into account the assumptions on the regularity of the solution $u$ while utilizing standard approximation properties, leads to 
\begin{align}\label{NS8}
	\begin{split}
		&\|[\phi^{n+1};\phi^{n}]\|^2_G+\frac{\nu\Delta t}{4}	\|\nabla \phi^{n+1}\|^2+\frac{\nu\Delta t}{4}	\|\nabla \phi^{n+1}\|^2+\frac{\nu\Delta t}{4}	\|\nabla \phi^{n}\|^2
		\\
		&\quad+\Delta t\left(\frac{\nu}{4C_I^2 H^2}-C\nu^{-1}\left(h^{2k+2}+h^{2k}+\|u^{n+1}\|^2_{L^{\infty}}+\|\nabla u^{n+1}\|^2_{L^3}\right)\right)\|\phi^{n+1}\|^2
		\\
		&\quad+\left(\frac{1}{2}-C\nu^{-1}\Delta t\left(h^{2k+2}+h^{2k}+\|u^{n+1}\|^2_{L^{\infty}}+\|\nabla u^{n+1}\|^2_{L^3}\right)\right)\|\phi^{n+1}-2\phi^{n}+\phi^{n-1}\|^2
		\\
		&\leq \|[\phi^{n};\phi^{n-1}]\|^2_G+\frac{\nu\Delta t}{4}	\|\nabla \phi^{n}\|^2+C\nu^{-1}\left(\Delta th^{2k+2}+\Delta th^{4k}+\Delta t^5\right).
	\end{split}
\end{align}
Choose $H$ and $\Delta t$ so that $\frac{\nu}{4C_I^2 H^2}-C\nu^{-1}\Delta t\left(h^{2k+2}+h^{2k}+\|u^{n+1}\|^2_{L^{\infty}}+\|\nabla u^{n+1}\|^2_{L^3}\right)>0$ and \newline $\frac{1}{2}-C\nu^{-1}\Delta t\left(h^{2k+2}+h^{2k}+\|u^{n+1}\|^2_{L^{\infty}}+\|\nabla u^{n+1}\|^2_{L^3}\right)>0$ are satisfied.  With these choices, using the Poincar\'{e} inequality, G-norm equivalence (\ref{G-norm}), CDA Stokes projection error estimate, and proceeding similarly as (\ref{ed4})-(\ref{ed44}), leads to
\begin{align}\label{NS9}
	\begin{split}
	\|[\phi^{n+1};\phi^{n}]\|^2_G+\frac{\nu\Delta t}{4}	\|\nabla \phi^{n+1}\|^2&\leq \left(\frac{1}{1+\lambda \Delta t}\right)^{n+1}\left(\|[\phi^{1};\phi^{0}]\|^2_{G}+\frac{\nu\Delta t}{4}	\|\nabla \phi^{1}\|^2\right)\\&~~~~+\frac{C\nu^{-1}\left(h^{2k+2}+h^{4k}+\Delta t^4\right)}{\lambda },
		\end{split}
\end{align}
where $\lambda=\min\left\{\Delta t^{-1},\frac{\nu C_l^2}{4C_P^{2}},\left(\frac{\nu}{4C_I^2 H^2}-C\nu^{-1}\left(h^{2k+2}+h^{2k}+\|u^{n+1}\|^2_{L^{\infty}}+\|\nabla u^{n+1}\|^2_{L^3}\right)\right)C_l^2\right\}$. \newline Then, the proof is completed by applying the triangle inequality.
\end{proof}

\section{Numerical Experiments} \label{sec:num-exp}

We now illustrate the effectiveness of the proposed method for several test problems.  For all of our tests, we take the $L^2$ projection onto piecewise constants as $I_H$, implemented with algebraic nudging \cite{RZ21}.

\subsection{The heat equation}

%\subsubsection{Convergence rate test}

For our first test, we pick a problem with domain $\Omega=(0,1)^2$, { $\kappa=1$}, and analytical solution
\[
u = \sin( t + 2\pi x + \pi y).  
\]
The solution then determines { $f=u_t-\Delta u$}, and the boundary conditions are taken to be Dirichlet on all of $\partial \Omega$ and equal to the solution.

We compute with the BDF2-FEM CDA algorithm above using $H = \frac19$ and $\mu=\infty$ (direct implementation) with $P_2$ elements on barycenter refinements of uniform triangular meshes.  We note that we also tested with $\mu=10^5,\ 10^8$ and saw no difference in results between these solutions and those found using direct implementation.

The theory above from Theorem \ref{thm1} shows that for $\mu$ sufficiently large and for $n$ large enough, 
\[
\| w_h^n - u(t^n) \| \le C(h^{k+1} + \Delta t^2).
\]
We illustrate this theory by first taking small $T=0.3$ and $\Delta t=0.001$, and varying $h$.  Errors and rates at the final time are shown in Table \ref{conv1}, and we observe third order convergence in space, which is consistent with the theory since we use $P_2$ elements.  To illustrate temporal convergence rates, we take $h=\frac{1}{256}$, $T=1$, and vary $\Delta t$.  Errors and rates at the final time are shown in Table \ref{conv2}, and we observe rates consistent with second order temporal convergence { (the last error and rate are reduced due to temporal error no longer being completely dominant over spatial error with the smallest choice of $\Delta t$).}  Plots of the error measured in the $L^2-$norm for all of these tests are shown in Figure \ref{fig1}, and the results are consistent with the theory - exponential convergence in time up to the discretization error, then long-time optimal accuracy, i.e.~$O(h^3 + \Delta t^2)$ since $P_2$ elements are used for this test, with respect to the $L^2-$norm.

\begin{figure}[ht]
\center
\includegraphics[width = .4\textwidth, height=.28\textwidth,viewport=0 0 550 400, clip]{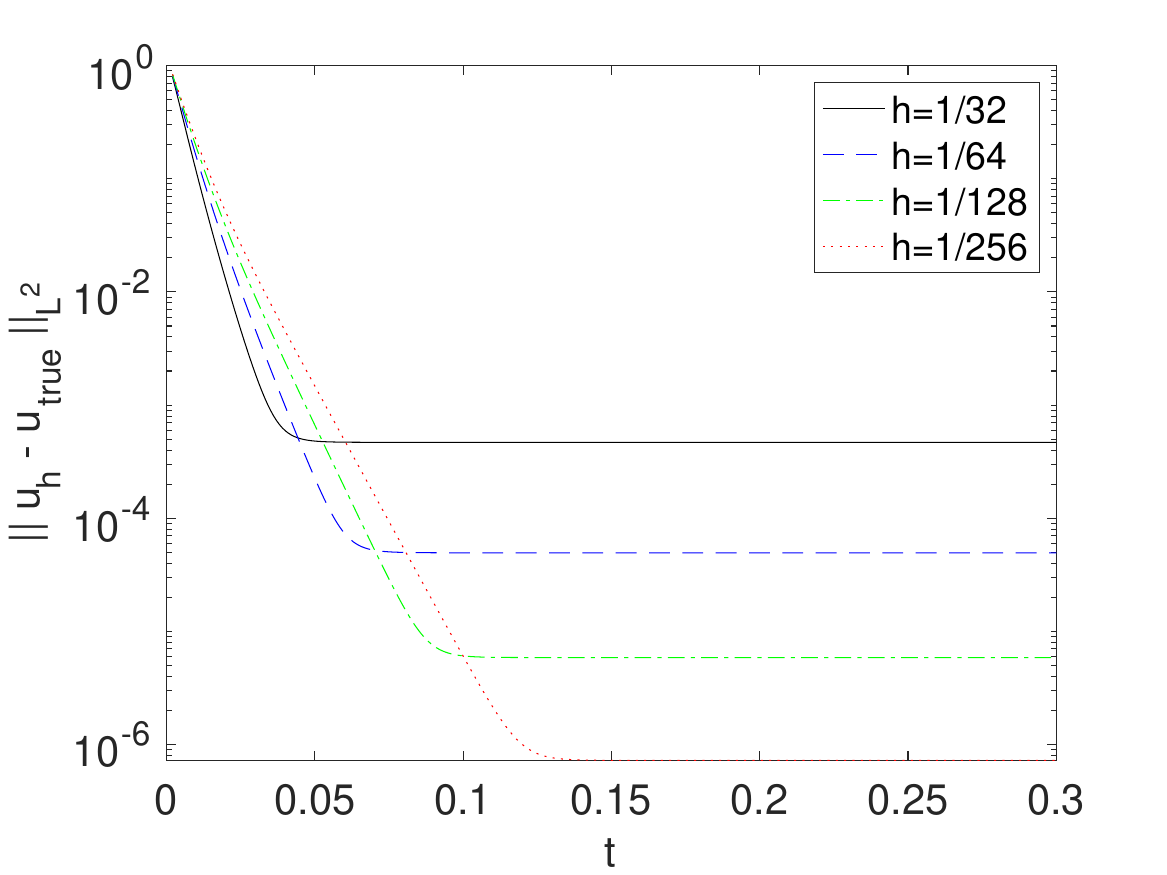}  
\includegraphics[width = .4\textwidth, height=.28\textwidth,viewport=0 0 550 400, clip]{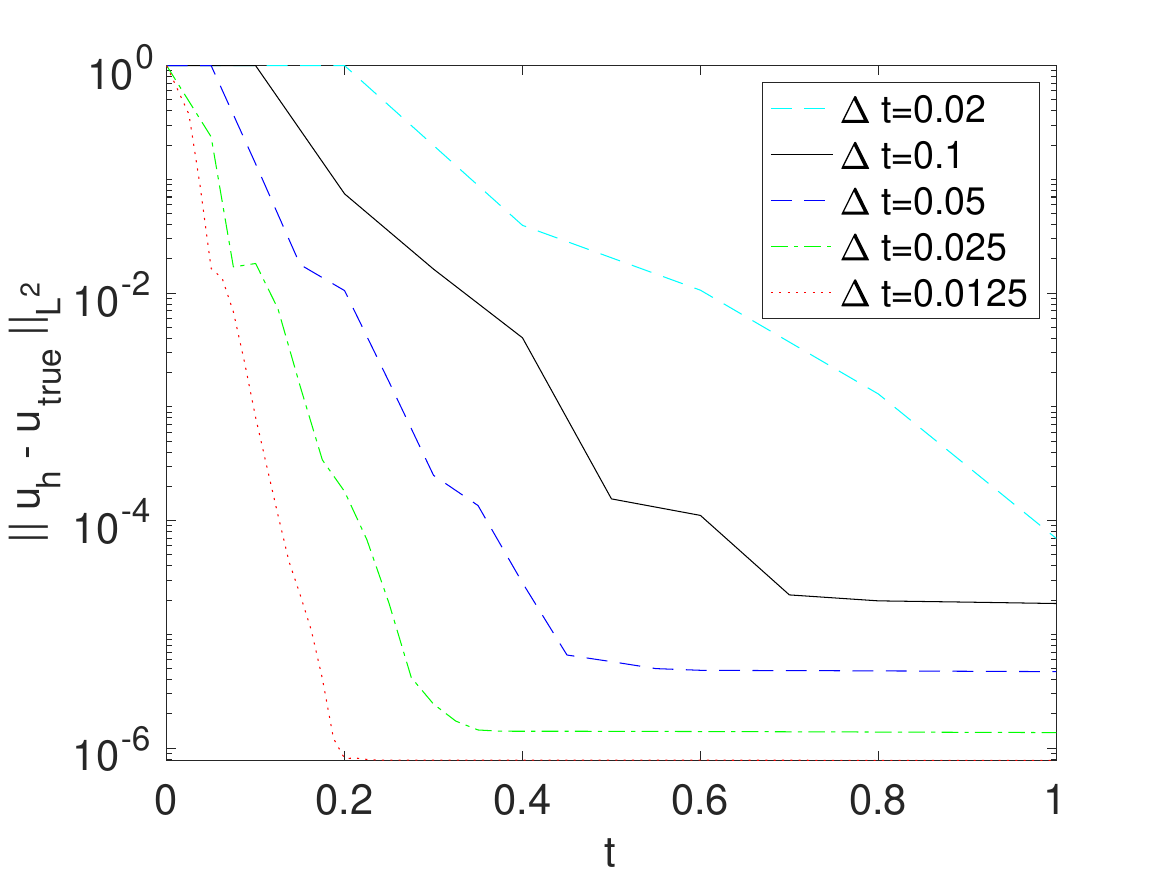}  
\caption{\label{fig1} Shown above are $L^2$ error versus time for the tests using the analytical solution for the heat equation and BDF2-CDA algorithm, for varying $h$ (left) and varying $\Delta t$ (right).}
\end{figure}

\begin{table}[H]
\centering
\begin{tabular}{|c|c|c|}
	\hline
	$h$ &  $\| w_h^M - u(0.3) \|$ & rate \\ \hline
1/32 &    4.690e-4 &  -   \\ \hline
1/64  &   4.947e-5 & 3.25   \\ \hline
1/128 &    5.865e-6 & 3.08 \\ \hline
1/256  & 7.236e-7 & 3.02   \\ \hline
\end{tabular}
\caption{Shown above are errors at the final time and spatial convergence rates for $T=0.3$ and $\Delta t=0.001$ ($M=T/\Delta t$).}\label{conv1}
\end{table}

\begin{table}[H]
\centering
\begin{tabular}{|c|c|c|}
	\hline
	$\Delta t$ &  $\| w_h^M - u(1) \|$ & rate \\ \hline
0.2 &    6.985e-5 &  -   \\ \hline
0.1  &   1.870e-5 & 1.90   \\ \hline
0.05 &    4.712e-6 & 1.99 \\ \hline
0.025  & 1.372e-6 & 1.78   \\ \hline
\end{tabular}
\caption{Shown above are errors and temporal convergence rates for $h=\frac{1}{256}$, $T=1.0$ and varying $\Delta t$ ($M=T/\Delta t$).}\label{conv2}
\end{table}

\subsection{A contaminant transport model}

In this numerical experiment, a CDA-FEM of the heat equation with added linear transport that models contaminant transport in a river is investigated. The linear transport equation takes the form
\begin{eqnarray}
u_t + U \cdot \nabla u - \kappa \Delta u &= &f, \label{ft1} \\
u(0) & = & u_0, \label{ft2}
\end{eqnarray}
where $U$ is a given divergence free velocity field, $\kappa$ is the (material dependent) diffusion coefficient, and here $u$ represents the concentration of a contaminant.  This system is similar to heat equation with an additional linear term $U\cdot\nabla u$, and thus CDA can be applied in the same way as above for the heat equation with (essentially) the same convergence results.

We follow a test problem from \cite{HJK18}, the domain for which is the area enclosed by the curves $y=\sin(x)$, $y=1+\sin(x)$, $x=0$ and $x=4\pi$ as left and right boundaries.  Using the mesh $\tau_h$ shown in Figure
\ref{meshes}, we use $(P_2,P_1)$ Taylor-Hood velocity-pressure elements (which gives 13,928 degrees of freedom (dof)) to compute the Stokes equations with viscosity $0.01$ and $f=0$, using no-slip boundary conditions on the top and bottom boundaries, a plug inflow of $u_{in}=3$, and a zero-traction outflow enforced with the do-nothing condition.  We take our transport velocity $U$ to be the velocity solution of this discrete Stokes problem.

\begin{figure}[!ht]
\begin{center}
\includegraphics[width = .48\textwidth, height=.15\textwidth,viewport=90 29 640 290, clip]{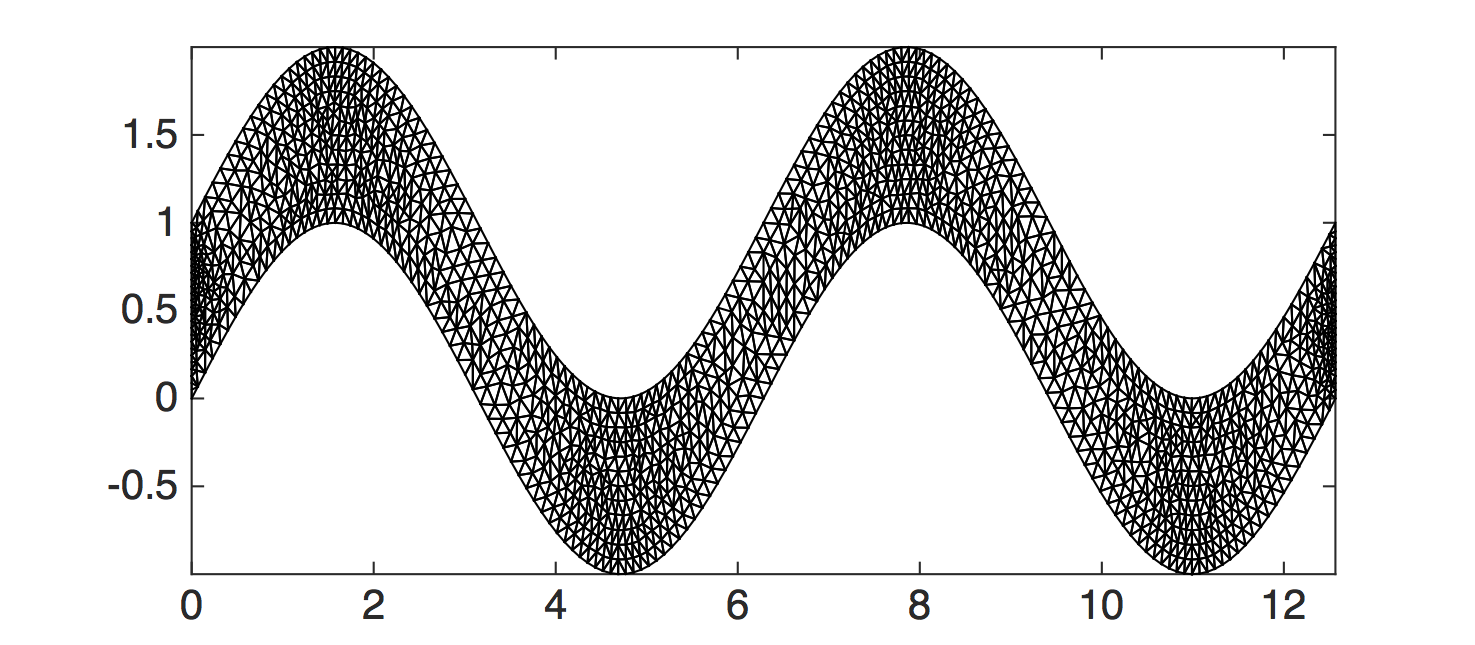}
\includegraphics[width = .48\textwidth, height=.15\textwidth,viewport=100 28 710 240, clip]{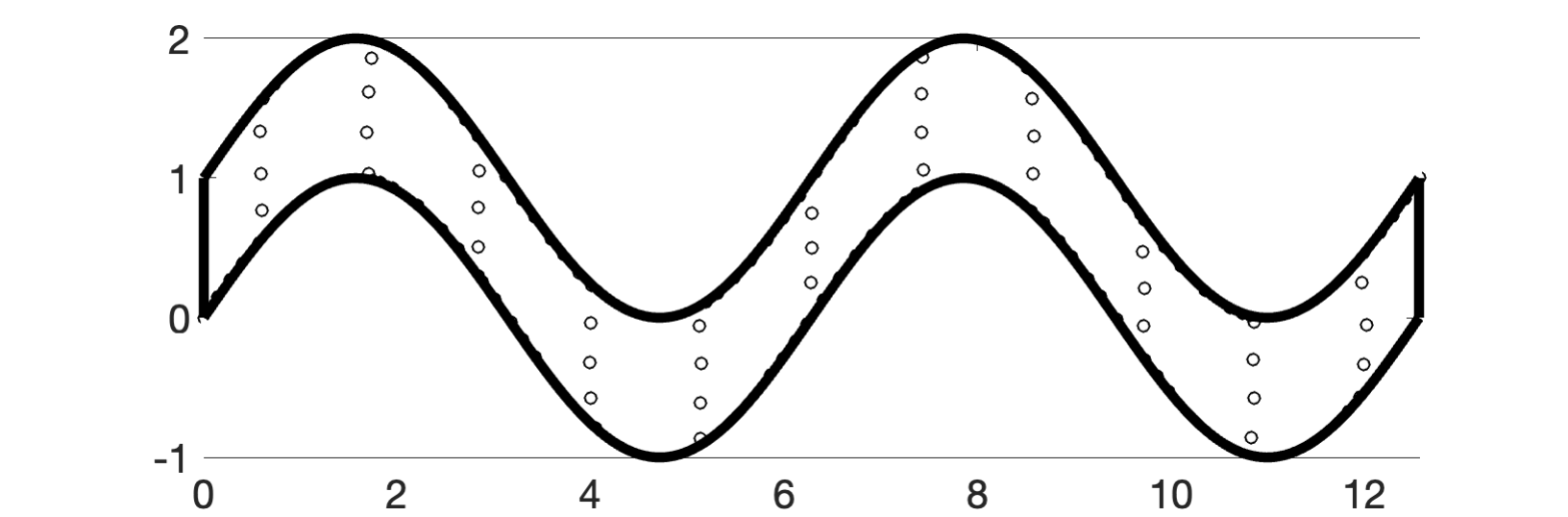}
	\caption{\label{meshes} Shown above are the fine mesh used for the finite element computations (left) and the coarse mesh nodes used for CDA (right).}
	\end{center}
\end{figure}

\begin{figure}[!ht]
\begin{center}
\includegraphics[width = .45\textwidth, height=.28\textwidth,viewport=0 0 710 360, clip]{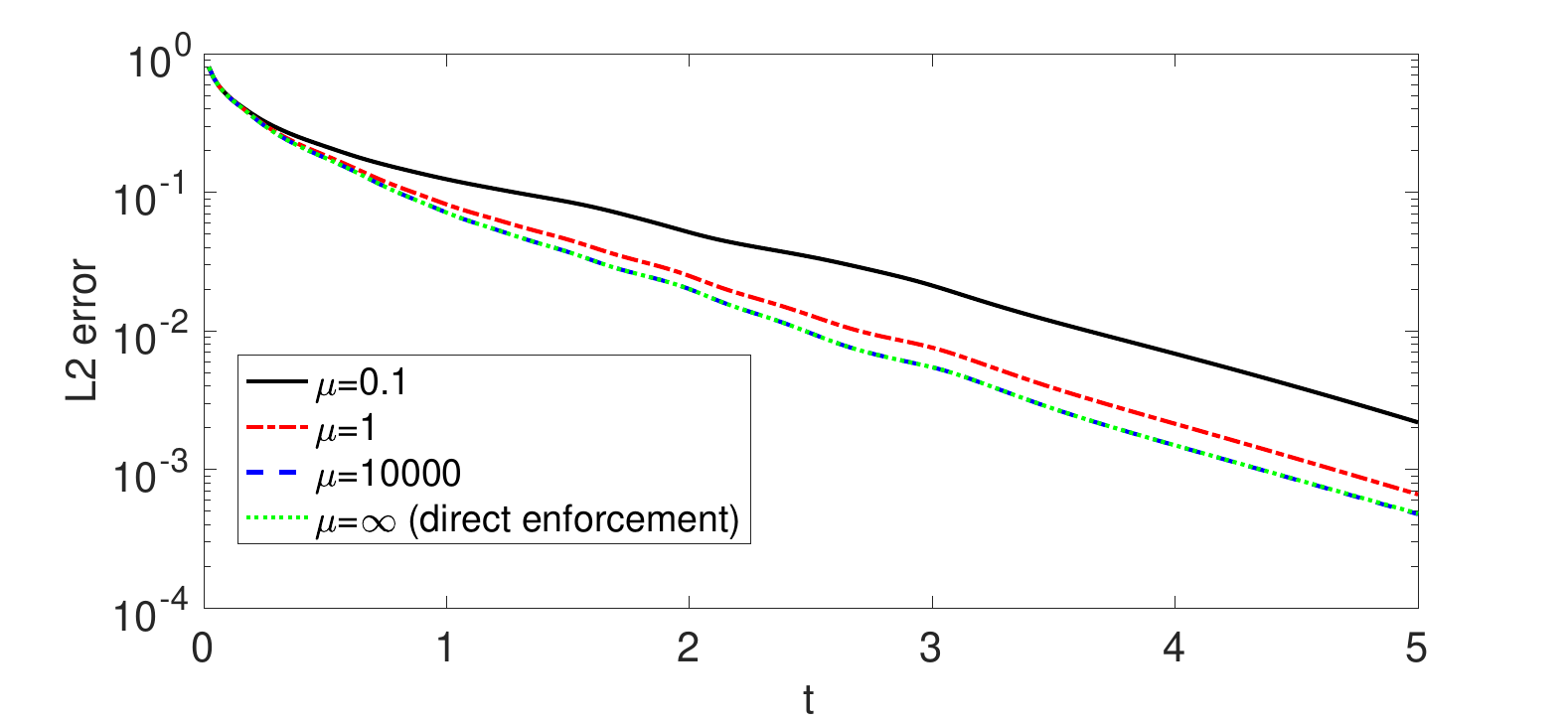}\\
	\caption{\label{conv} Shown above are plots of convergence of the BDF2-CDA solutions to the true solution, for varying $\mu$.}
	\end{center}
\end{figure}

We use BDF2-FEM with CDA to solve \eqref{ft1}-\eqref{ft2} equipped with a boundary condition of 0 at the inflow ($c_{in}=0$), the transport velocity $U$ coming from the discrete Stokes problem, transport term discretization 
\[
b^*(U,w_h^{n+1},v):= b(U,w_h^{n+1},v_h) + \frac12 ((\nabla \cdot U)w_h^{n+1},v_h)
\]
where $b(U,w_h^{n+1},v_h) := (U\cdot\nabla w_h^{n+1},v_h)$, $\kappa=0.01$, and 
zero Neumann conditions  $\nabla c \cdot n =0$ at all other boundaries.  For the initial condition, we take zero contaminant except for two `blobs', which are represented with $c=3$ inside the circles  with radius $0.1$ centered at $(1,1.5)$ and $(5,-0.5)$.  We first compute a direct numerical simulation (DNS) for the concentration $c$ using BDF2-FEM, $P_2$ on the same mesh as is used for the Stokes problem, and time step $\Delta t=0.02$.

\begin{figure}[!ht]
\begin{center}
DA (t=0.0) \hspace{2in} DNS (t=0.0)\\
\includegraphics[width = .48\textwidth, height=.13\textwidth,viewport=90 20 750 290, clip]{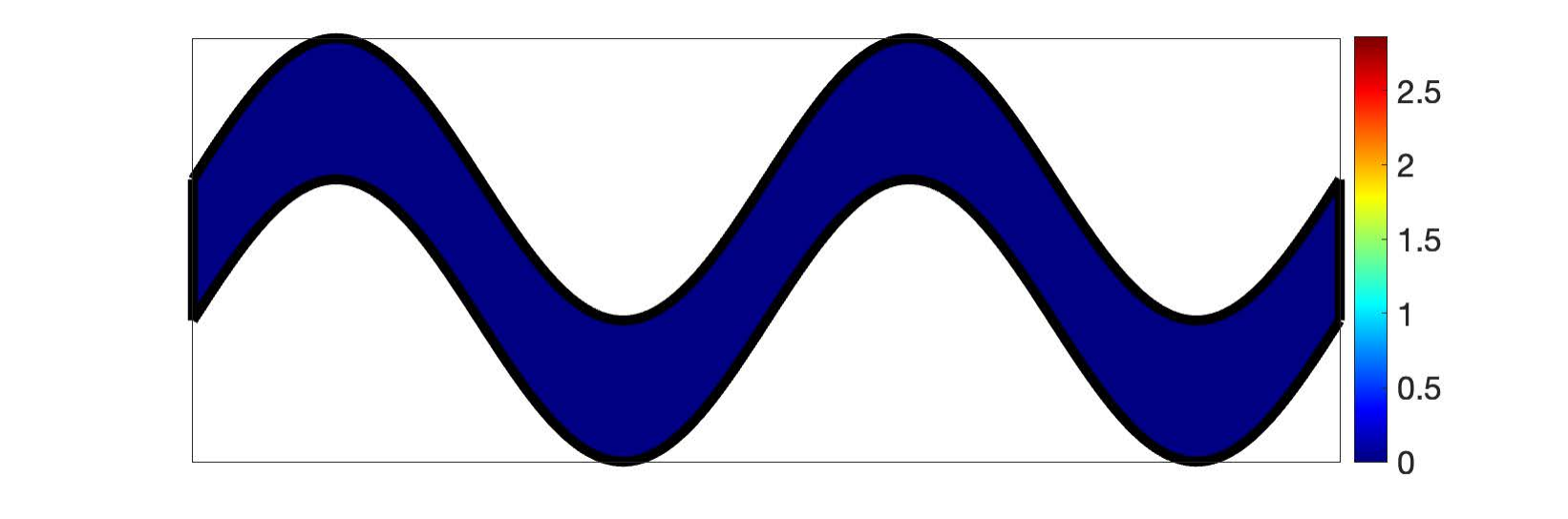}
\includegraphics[width = .48\textwidth, height=.13\textwidth,viewport=90 20 750 290, clip]{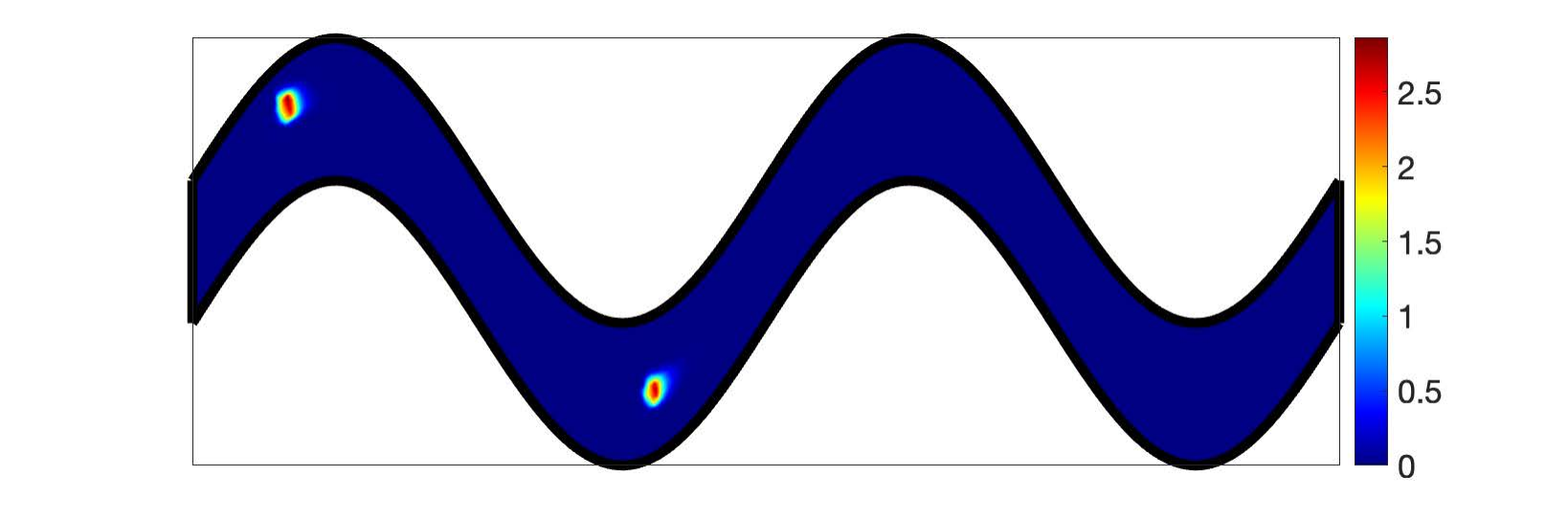}
DA (t=1.0) \hspace{2in} DNS (t=1.0)\\
\includegraphics[width = .48\textwidth, height=.13\textwidth,viewport=90 20 750 290, clip]{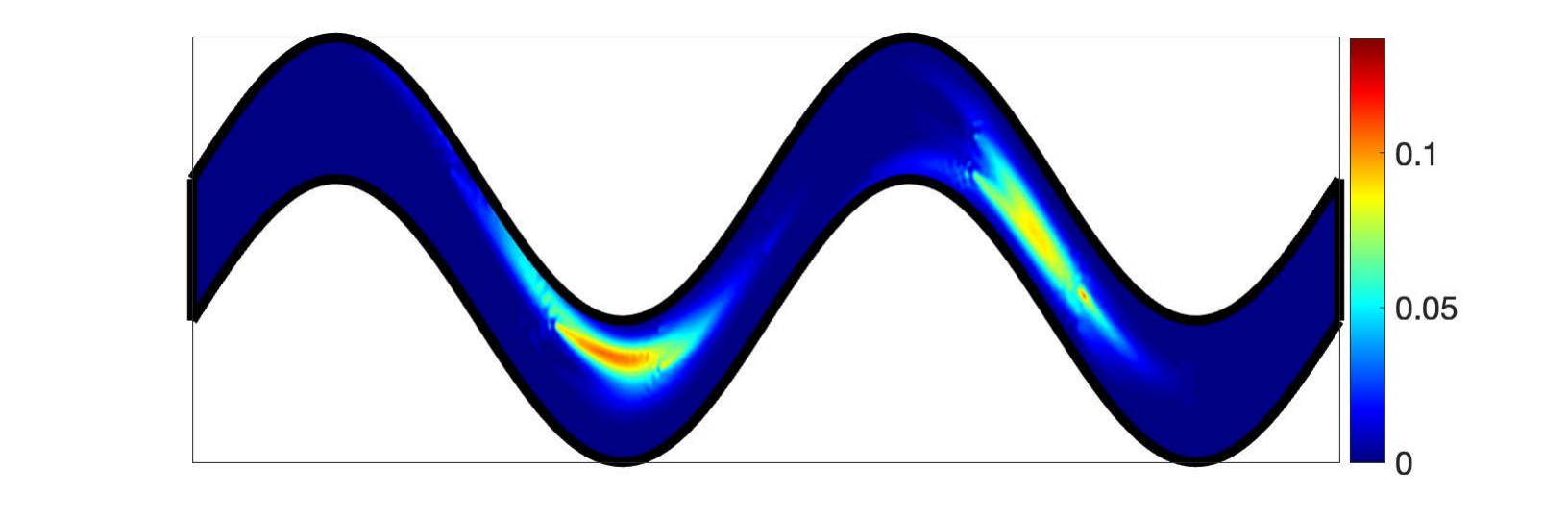}
\includegraphics[width = .48\textwidth, height=.13\textwidth,viewport=90 20 750 290, clip]{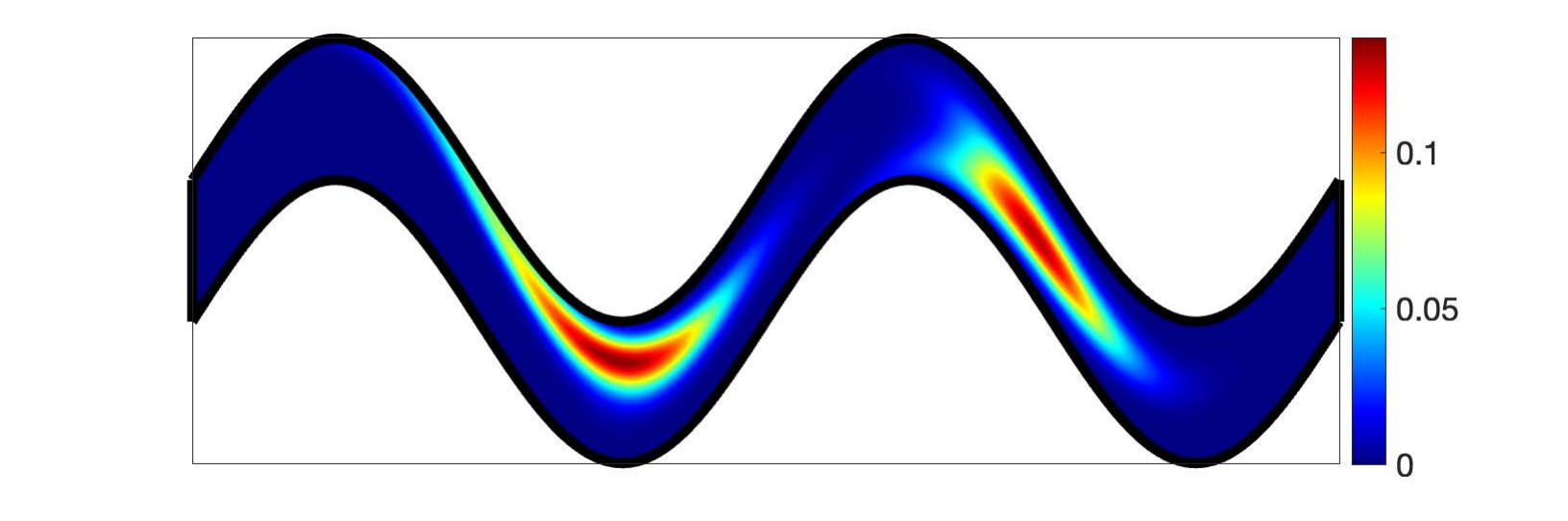}
DA (t=2.0) \hspace{2in} DNS (t=2.0)\\
\includegraphics[width = .48\textwidth, height=.13\textwidth,viewport=90 20 750 290, clip]{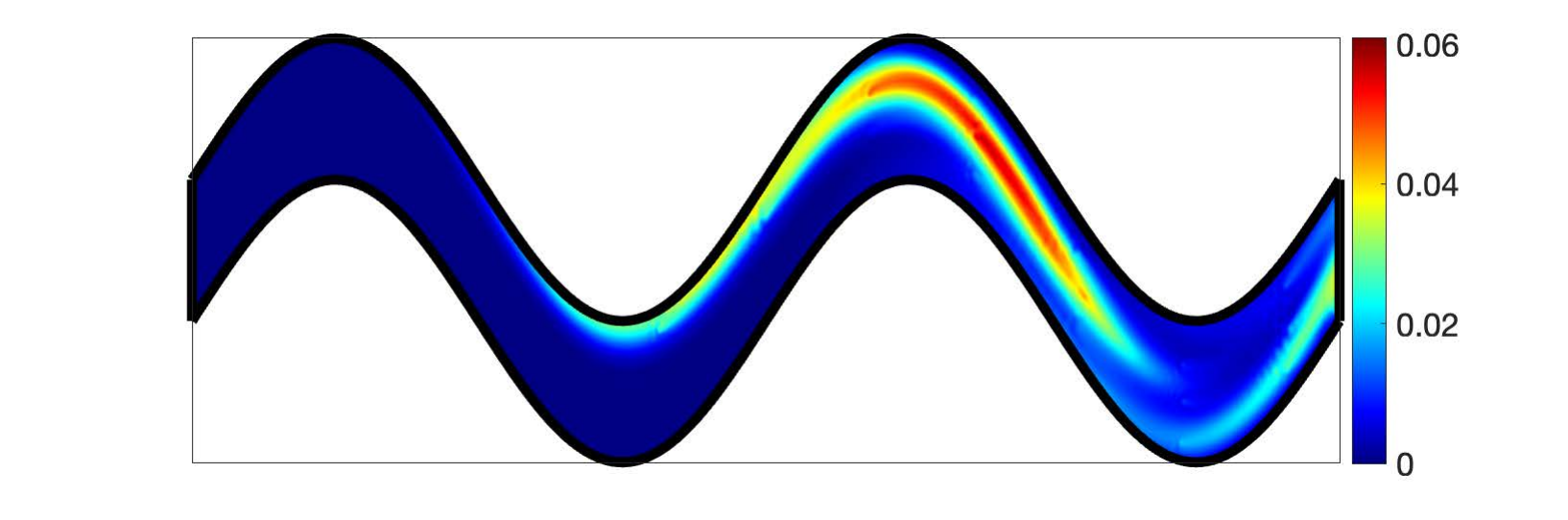}
\includegraphics[width = .48\textwidth, height=.13\textwidth,viewport=90 20 750 290, clip]{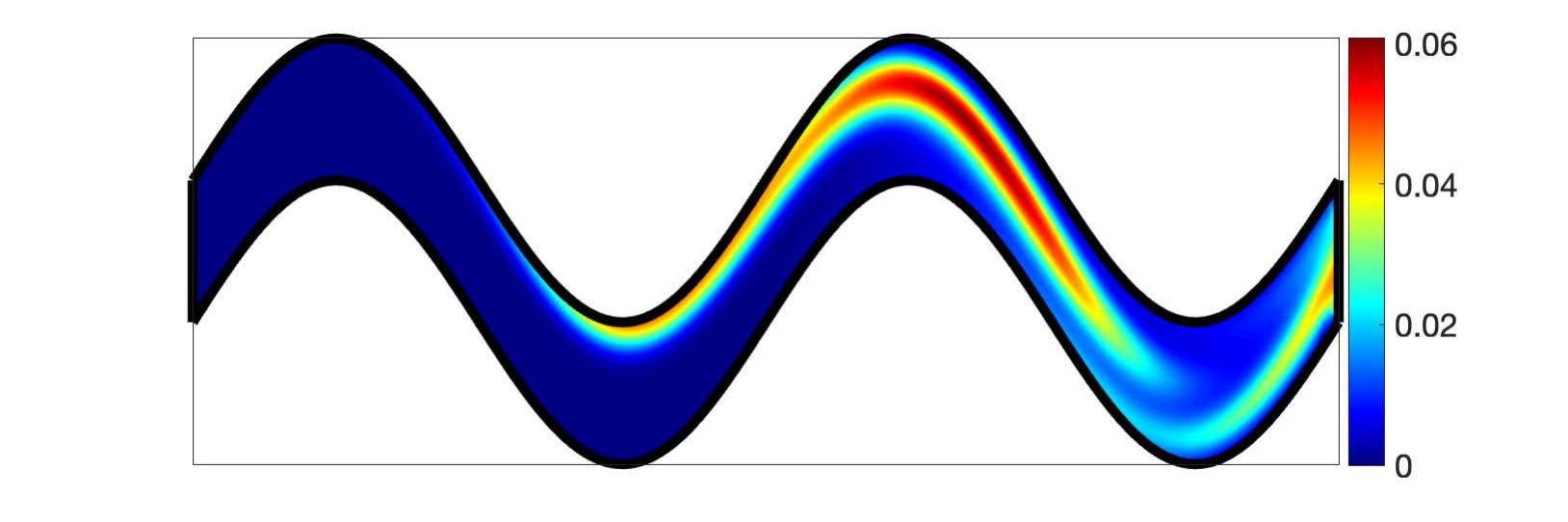}
DA (t=5.0) \hspace{2in} DNS (t=5.0)\\
\includegraphics[width = .48\textwidth, height=.13\textwidth,viewport=90 20 750 290, clip]{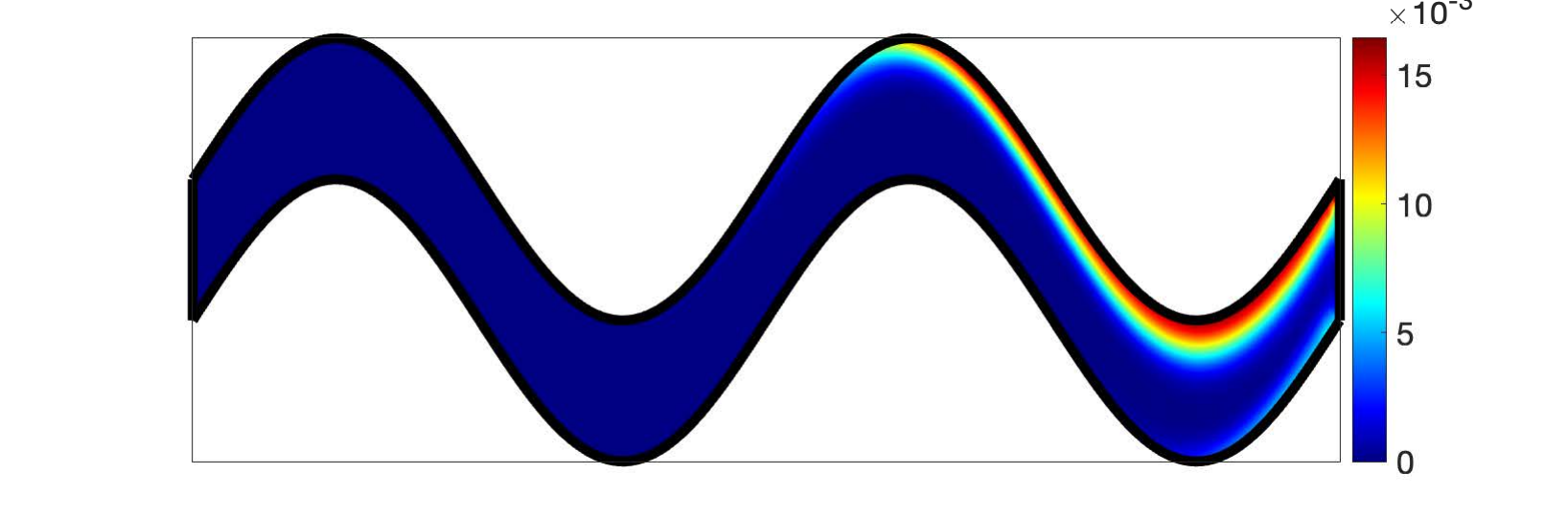}
\includegraphics[width = .48\textwidth, height=.13\textwidth,viewport=90 20 750 290, clip]{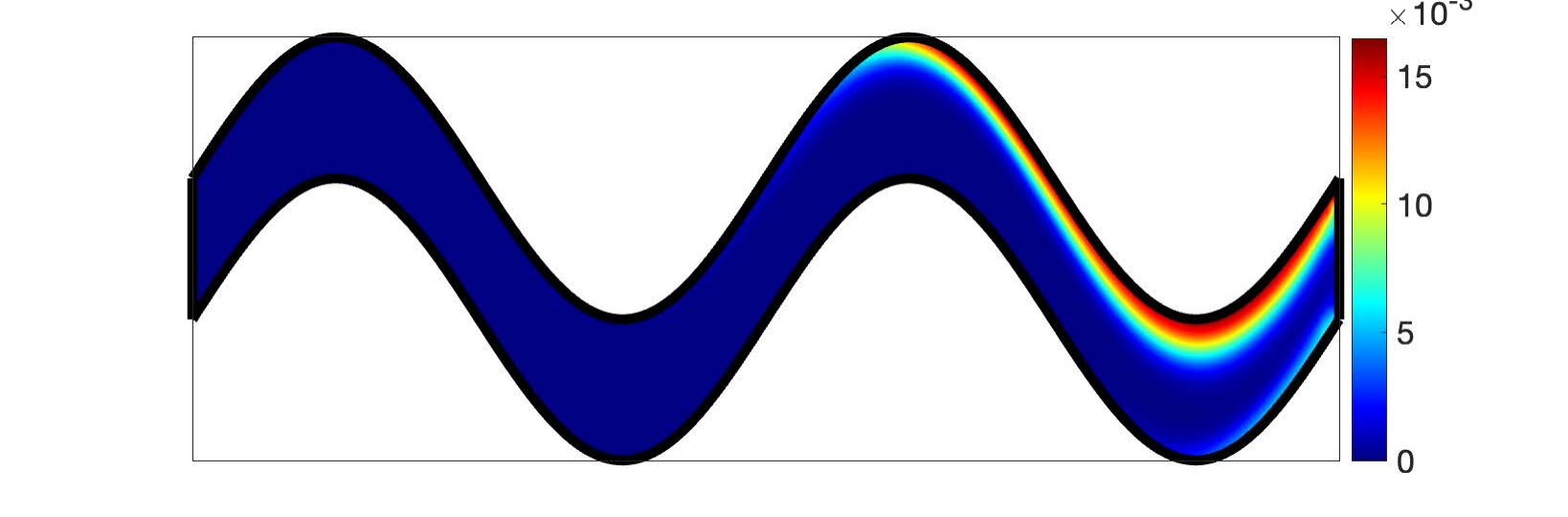}
	\caption{\label{dacontour} Contour plots of DA and DNS velocity magnitudes at times 0, 1, 2 and 5.}
	\end{center}
\end{figure}

Finally, we compute the CDA solution, using BDF2-CDA with the same parameters as the DNS except zero initial conditions, taking the DNS solution $u$ as the true solution, and testing the algorithm using several choices of $\mu$ including $\mu=\infty$ which we implement with direct enforcement.  The 37 measurement nodes are shown in Figure \ref{meshes} and we note that these nodes are also vertices on the finite element mesh.  Convergence of the CDA solution to the true (DNS) solution is shown in Figure \ref{conv}, for varying $\mu$.  We observe very similar convergence for each choice of $\mu$=0.1, 1, 10000, and $\infty$, with the plots for $\mu=10000$ and $\infty$ (direct enforcement) laying on top of each other.  For the case $\mu=\infty$, we show contour plots of the DA and true (DNS) solution at t=0, 1, 2, and 5 in Figure \ref{dacontour}, and observe that the agreement between the solutions increases as time-stepping moves forward and that the solutions are visually indistinguishable at the final stopping time.

\subsection{The Navier-Stokes equations}

\begin{figure}[!ht]
\begin{center}
\includegraphics[width = .6\textwidth, height=.4\textwidth,viewport=0 0 530 400, clip]{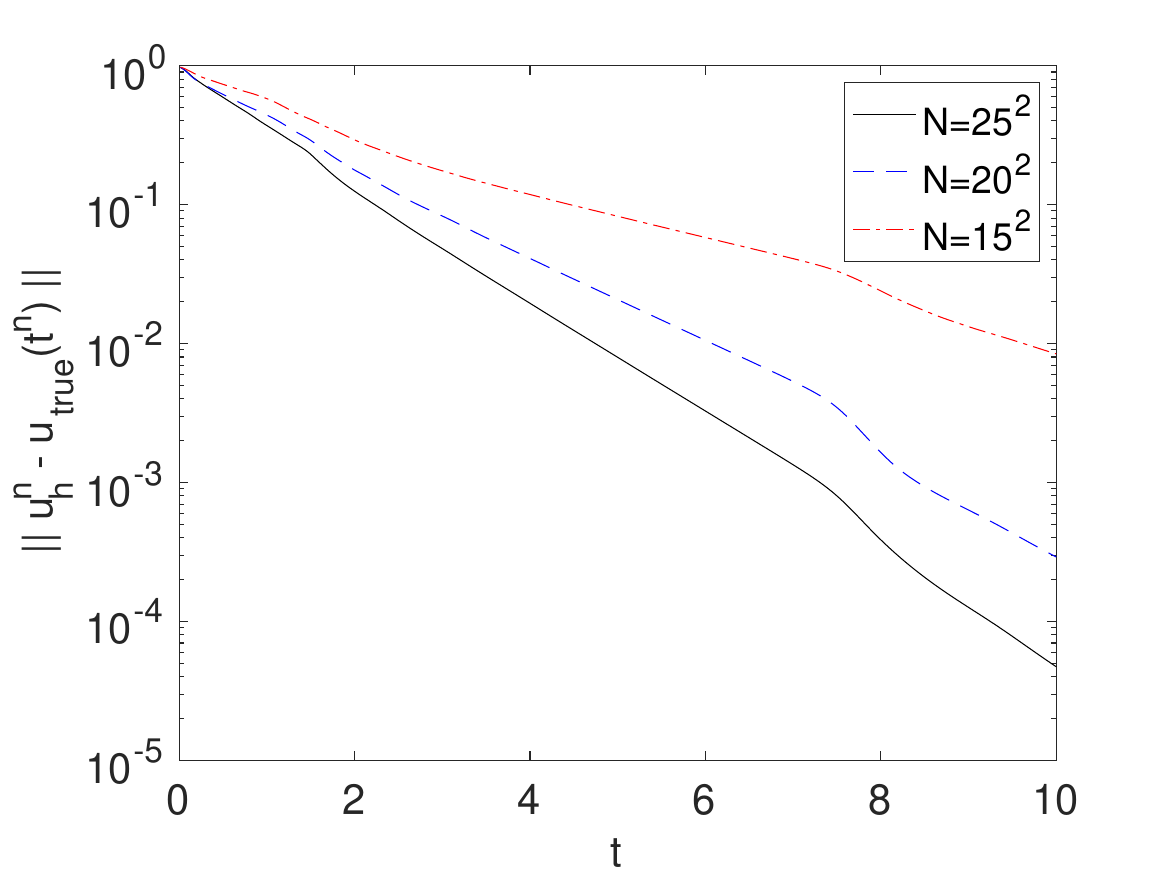}
	\caption{\label{KHplots1} Shown above are convergence plots of the CDA solution to the true solution for the Kelvin-Helmholtz test problem.}
	\end{center}
\end{figure}

To illustrate the theory for the Navier-Stokes equations, we choose a test problem from \cite{SJLLLS18,OR24} for simulating 2D Kelvin-Helmholtz instability.  The domain is taken to be the unit square and enforce periodic
boundary conditions on $x=0,1$ (representing an infinite extension in the horizontal direction).  At $y=0,1$, we enforce a no penetration and free slip condition.  The initial condition for the benchmark problem is set as
\[
u_0(x,y) = \left( \begin{array}{c} u_{\infty} \tanh\left( \frac{2y-1}{\delta_0} \right) \\ 0 \end{array} \right) + c_n \left( \begin{array}{c} \partial_y \psi(x,y) \\ -\partial_x \psi(x,y) \end{array} \right),
\]
where $\delta_0=\frac{1}{28}$ is the initial vorticity thickness, $u_{\infty}=1$ is set as the reference velocity, $c_n$ is a noise/scaling factor that is taken to be $10^{-3}$, and
\[
\psi(x,y) = u_{\infty} \exp \left( -\frac{(y-0.5)^2 }{\delta_0^2} \right) \left( \cos(8\pi x) + \cos(20\pi x) \right).
\]
The Reynolds number for this test is defined by $Re=\frac{\delta_0 u_{\infty}}{\nu} = \frac{1}{28 \nu}$, and $\nu$ is defined by selecting $Re$.  We use $Re=$100 for this test.

We compute solutions for the BDF2-CDA for the NSE similar to the previous experiments: first compute a reference solution (using the initial condition above), and then use this reference solution as the true solution to nudge towards.  We use $(P_2,P_1^{disc})$ Scott-Vogelius elements on a barycenter refinement of an $h=1/96$ uniform triangular mesh, and a time step size of $\Delta t=0.01$.  For the CDA, we take the initial condition to be $w_h^0=w_h^1=0$, and compute using Algorithm \ref{alg1} with $\mu=\infty$ (implemented with direct enforcement) and varying $H$ (=1/25, 1/20, and 1/15) using $N=\frac{1}{H^2}$ measurement points.  Convergence of the CDA solutions to the true solutions are shown in Figure \ref{KHplots1}, and we observe that the method is converging in $L^2$ for all three choices of $H$ with faster convergence as $H$ decreases.  Contour plots of the solution's absolute vorticity magnitudes are shown in Figure \ref{KHplots2}, along with those of the true solution.  At $t=2$, none of the CDA solutions appear accurate.  By $t=5$, the N=20 and 25 solutions are already approaching the true solution, and by $t=8$ their solutions are visually indistinguishable from the true solution.  At $t=9$, there are still visible errors in the $N=15$ solution, demonstrating that the inclusion of more information yields faster convergence.

\begin{figure}[!ht]
\begin{center}
\ \ \ \ NSE  \hspace{.9in} DA (N=$15^2$)  \hspace{.7in} DA (N=$20^2$)  \hspace{.7in} DA (N=$25^2$) \\
\includegraphics[width = .24\textwidth, height=.22\textwidth,viewport=65 40 530 400, clip]{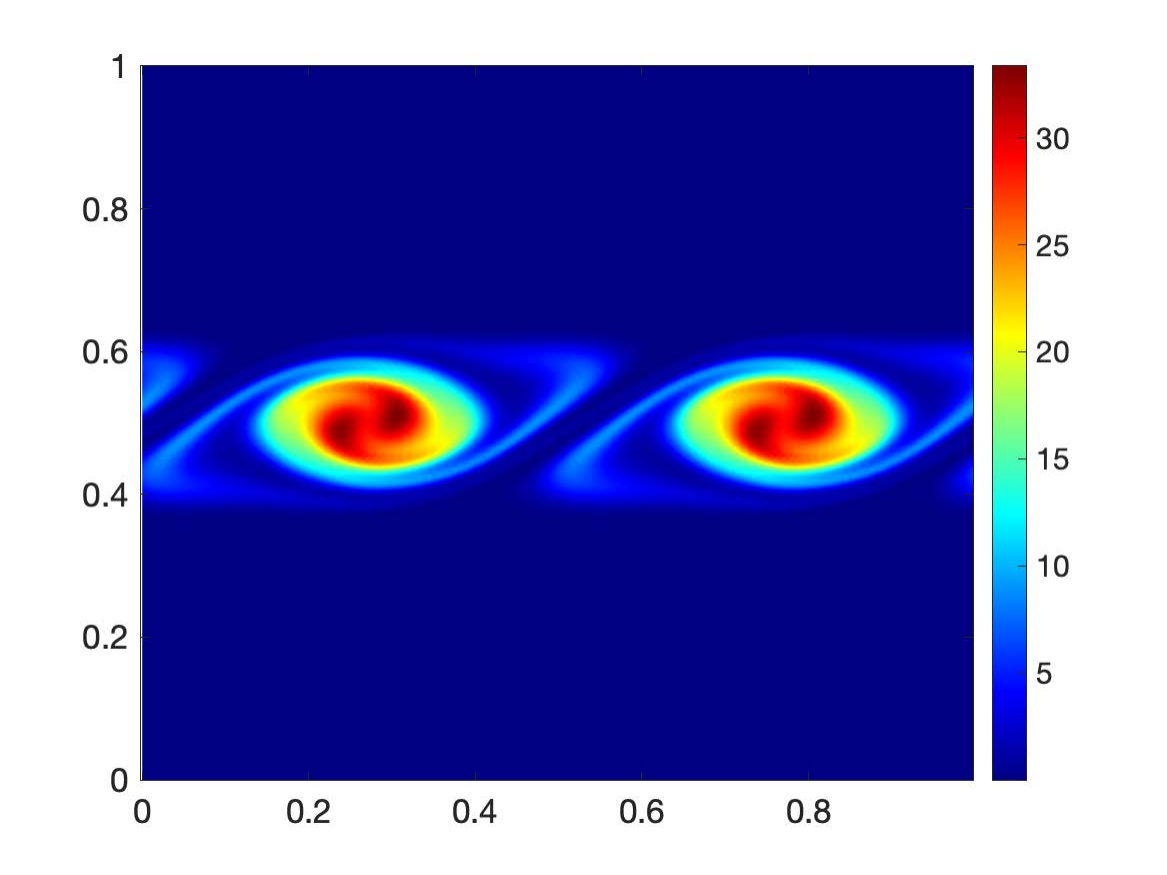}
\includegraphics[width = .24\textwidth, height=.22\textwidth,viewport=65 40 530 400, clip]{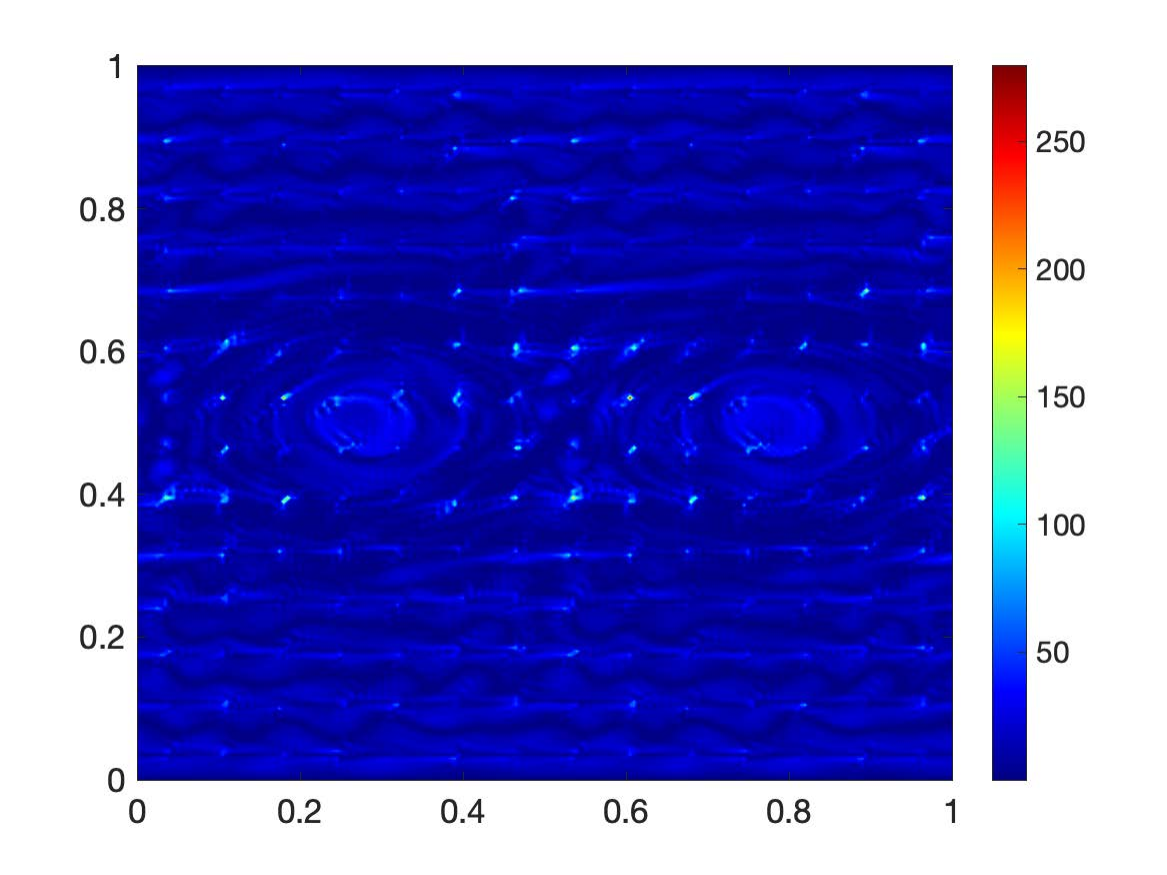}
\includegraphics[width = .24\textwidth, height=.22\textwidth,viewport=65 40 530 400, clip]{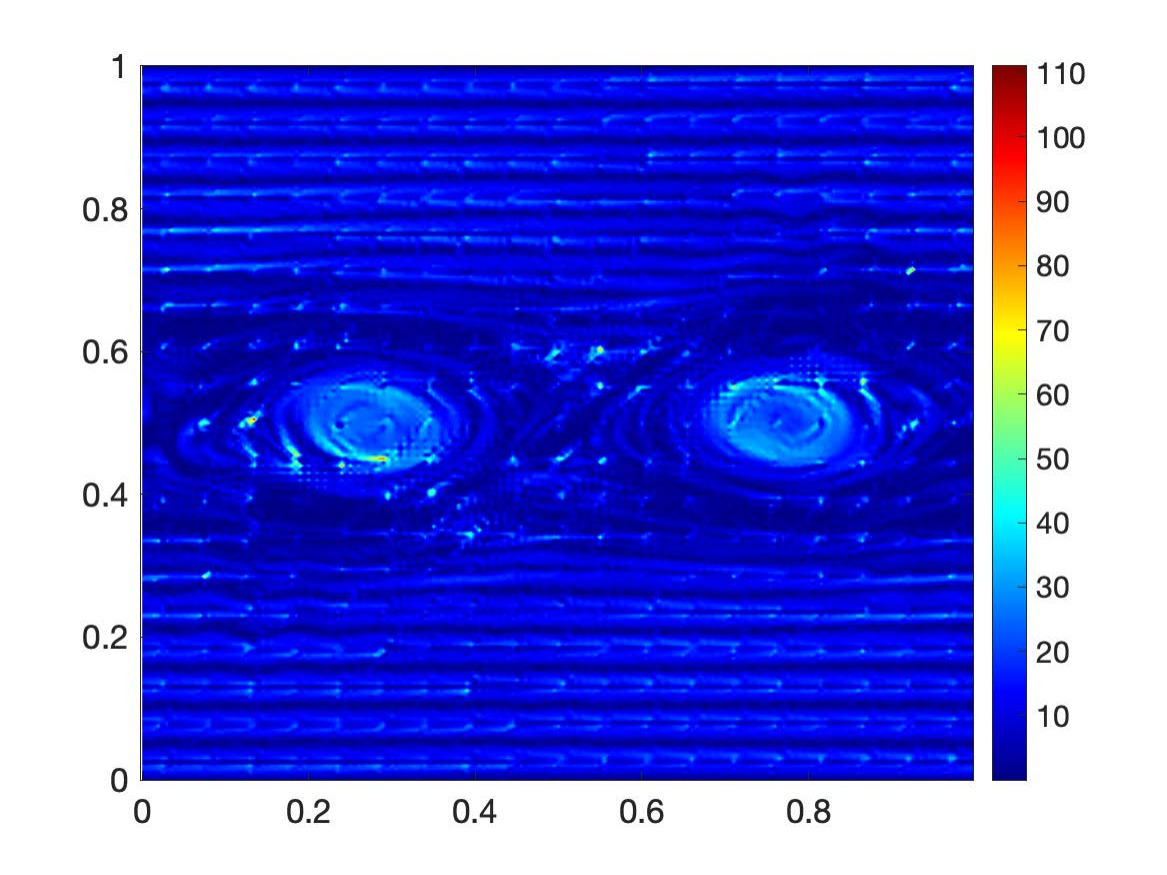}
\includegraphics[width = .24\textwidth, height=.22\textwidth,viewport=65 40 530 400, clip]{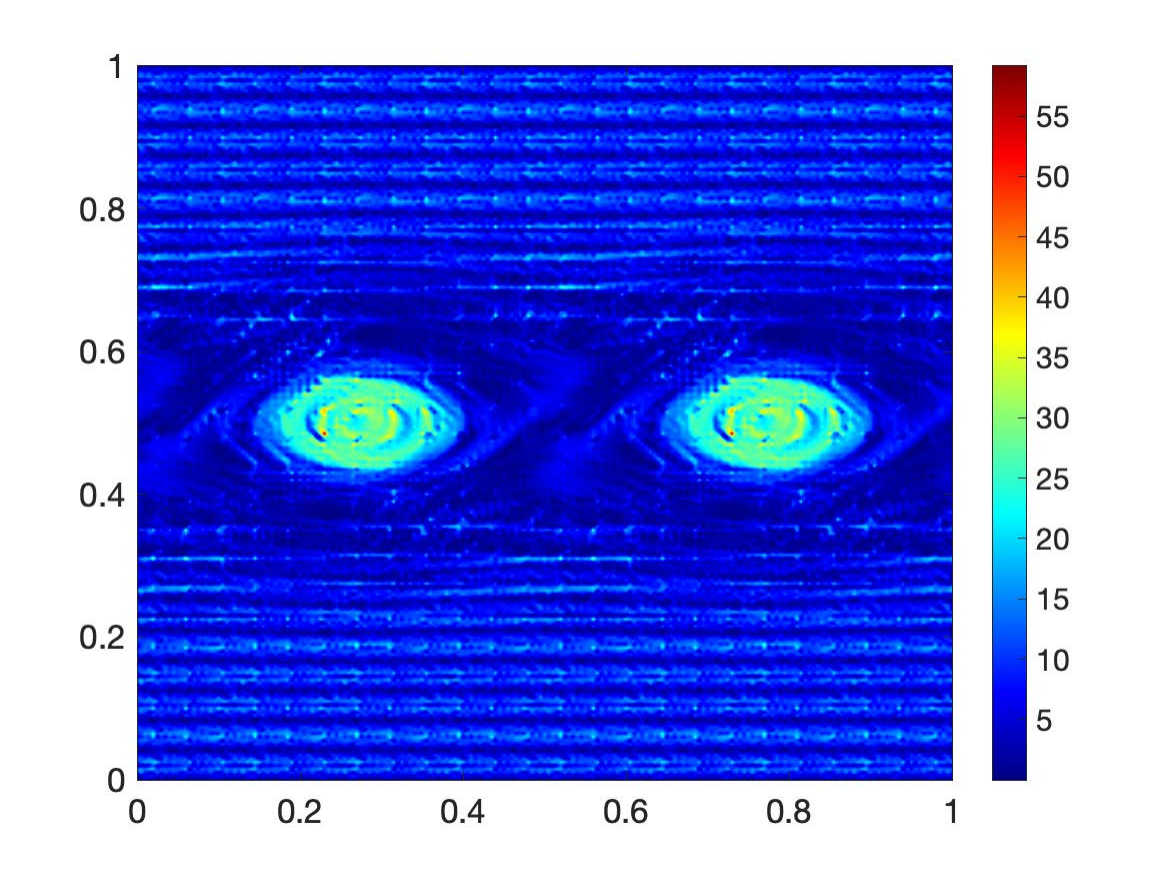}\\
\includegraphics[width = .24\textwidth, height=.22\textwidth,viewport=65 40 530 400, clip]{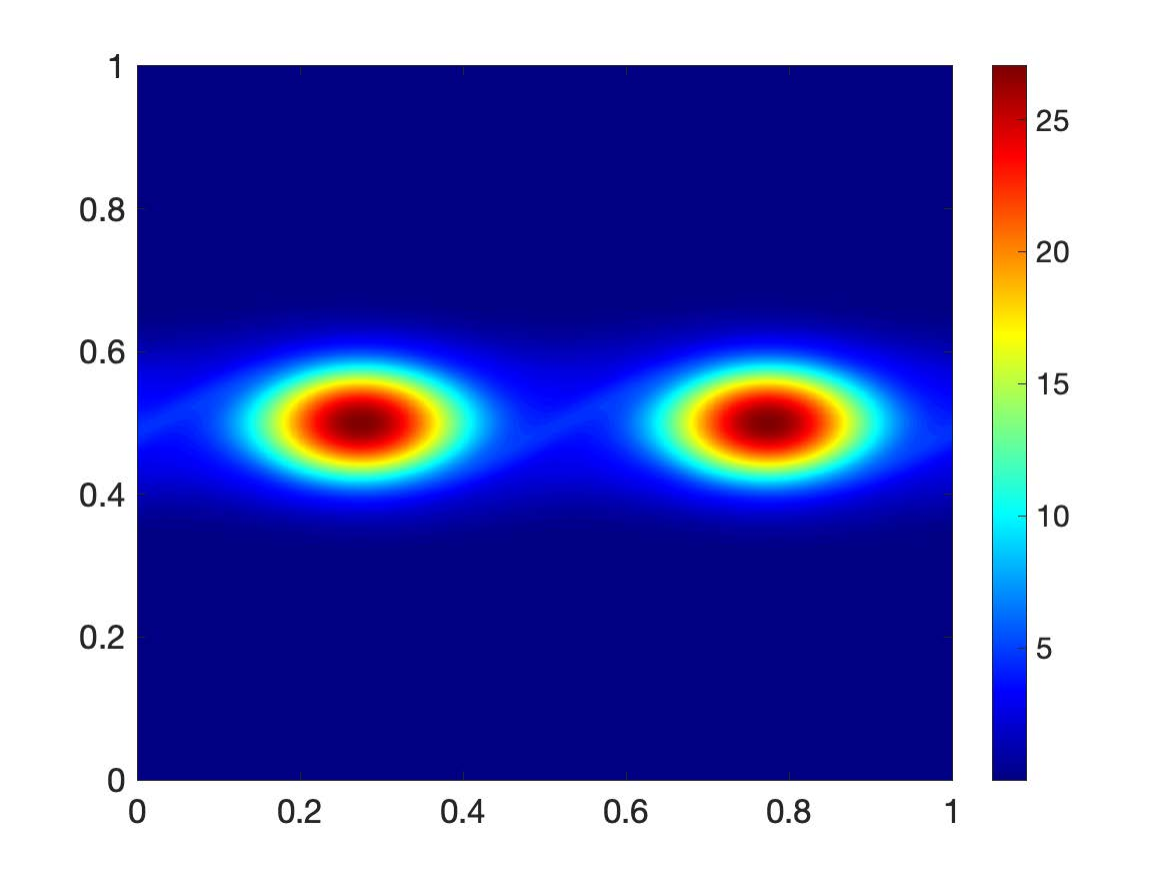}
\includegraphics[width = .24\textwidth, height=.22\textwidth,viewport=65 40 530 400, clip]{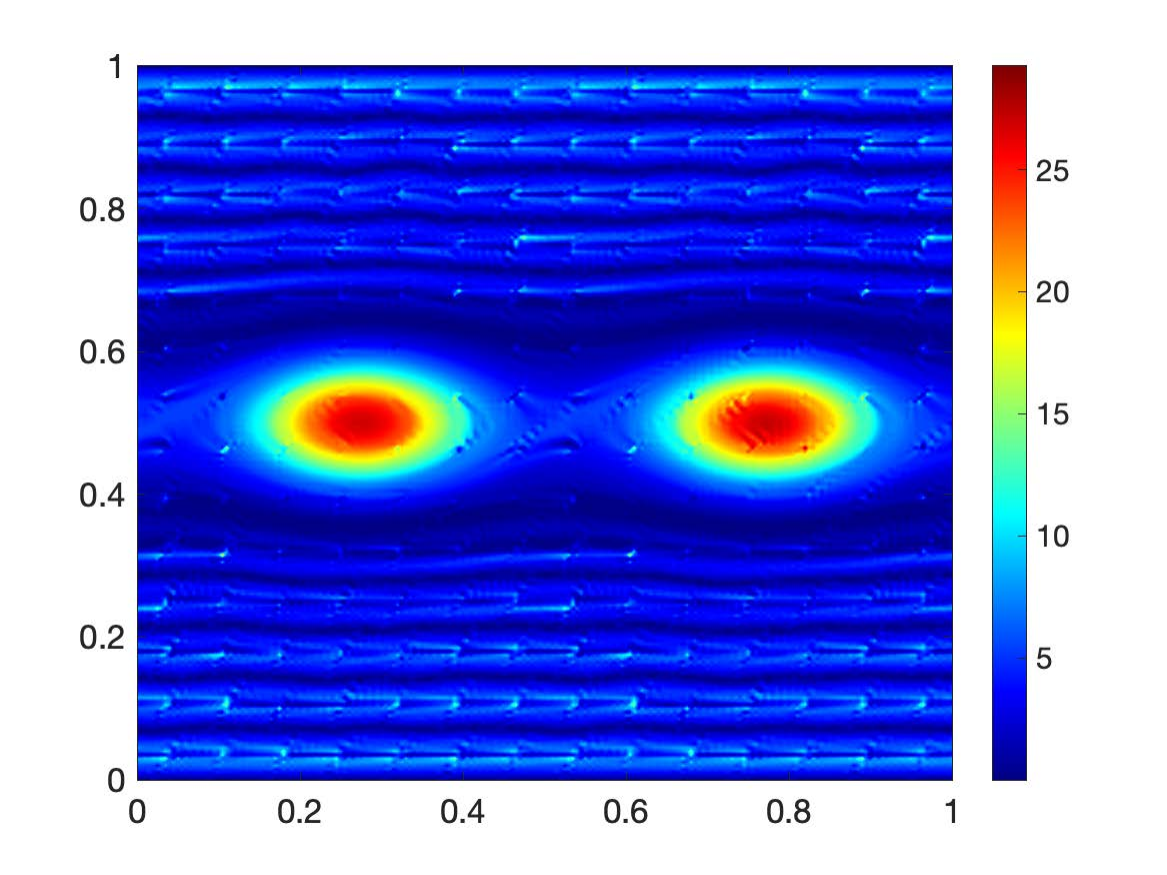}
\includegraphics[width = .24\textwidth, height=.22\textwidth,viewport=65 40 530 400, clip]{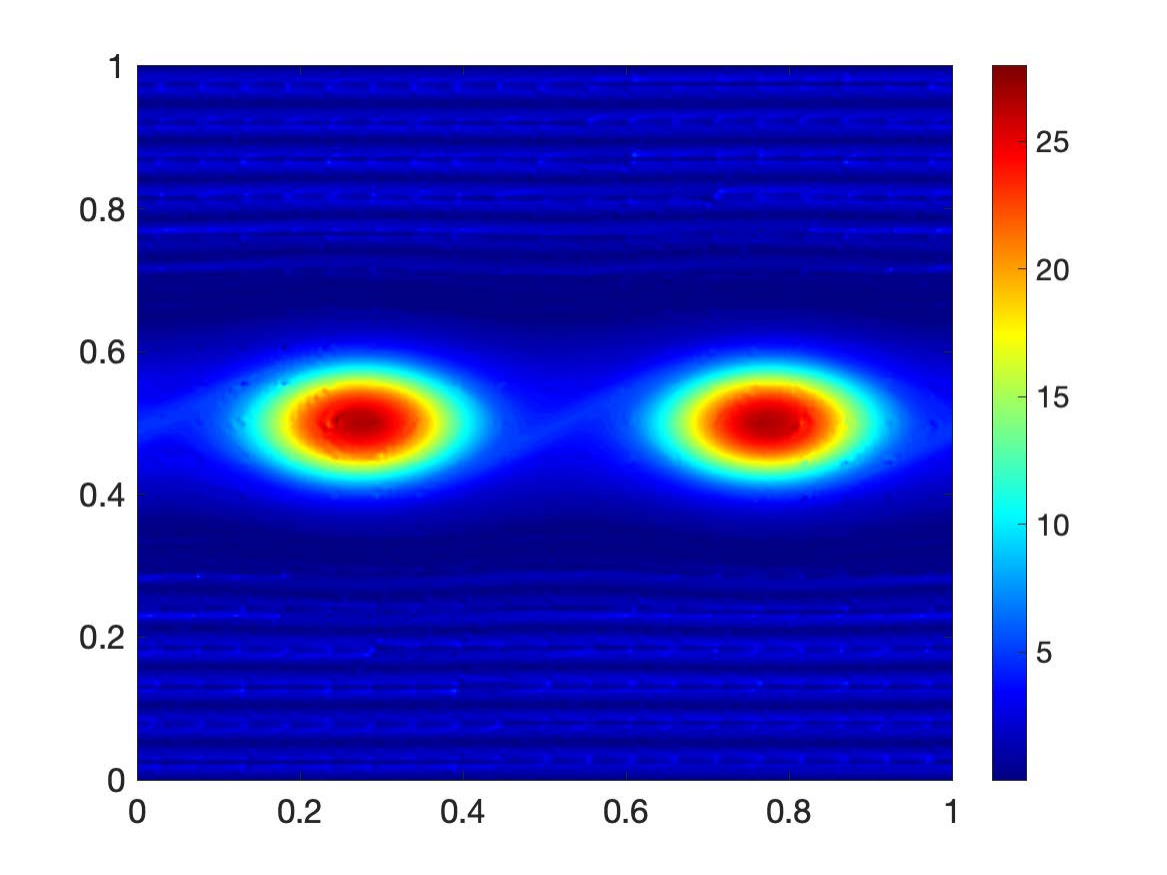}
\includegraphics[width = .24\textwidth, height=.22\textwidth,viewport=65 40 530 400, clip]{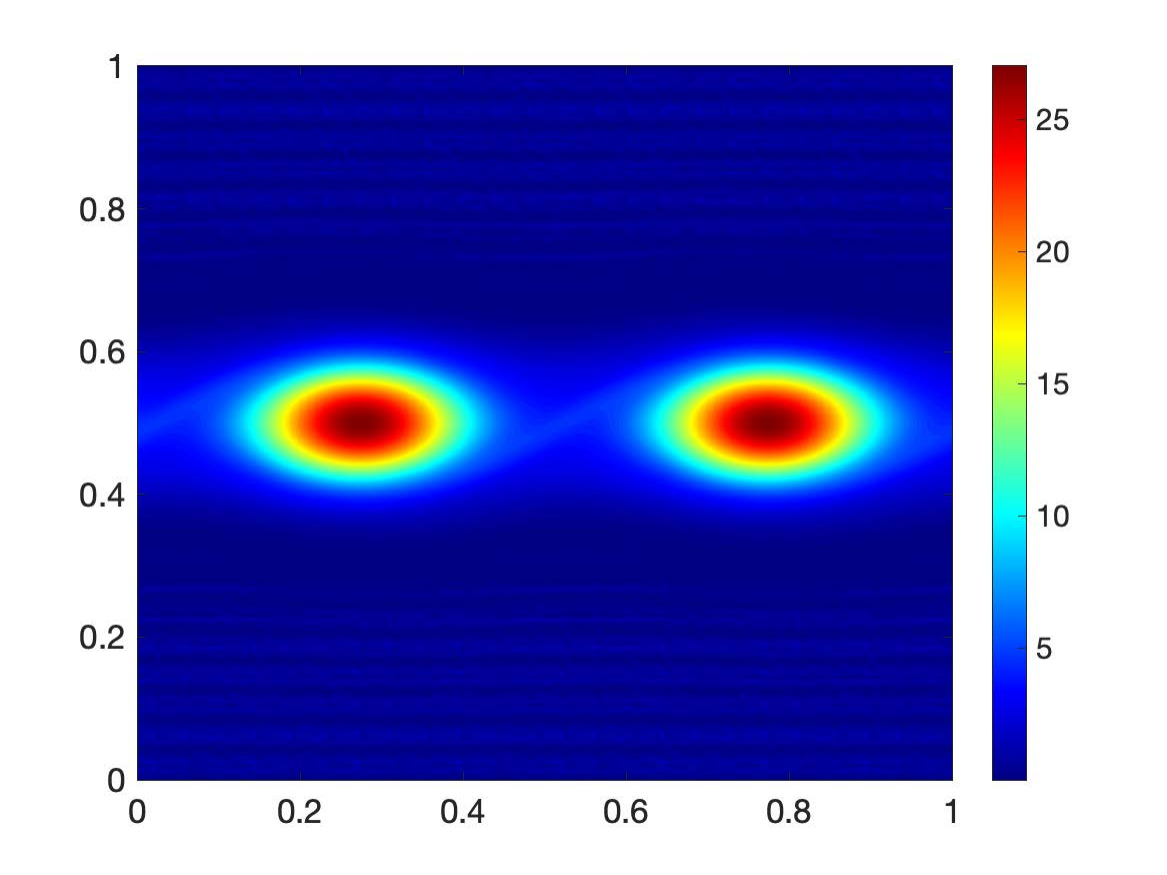}\\
\includegraphics[width = .24\textwidth, height=.22\textwidth,viewport=65 40 530 400, clip]{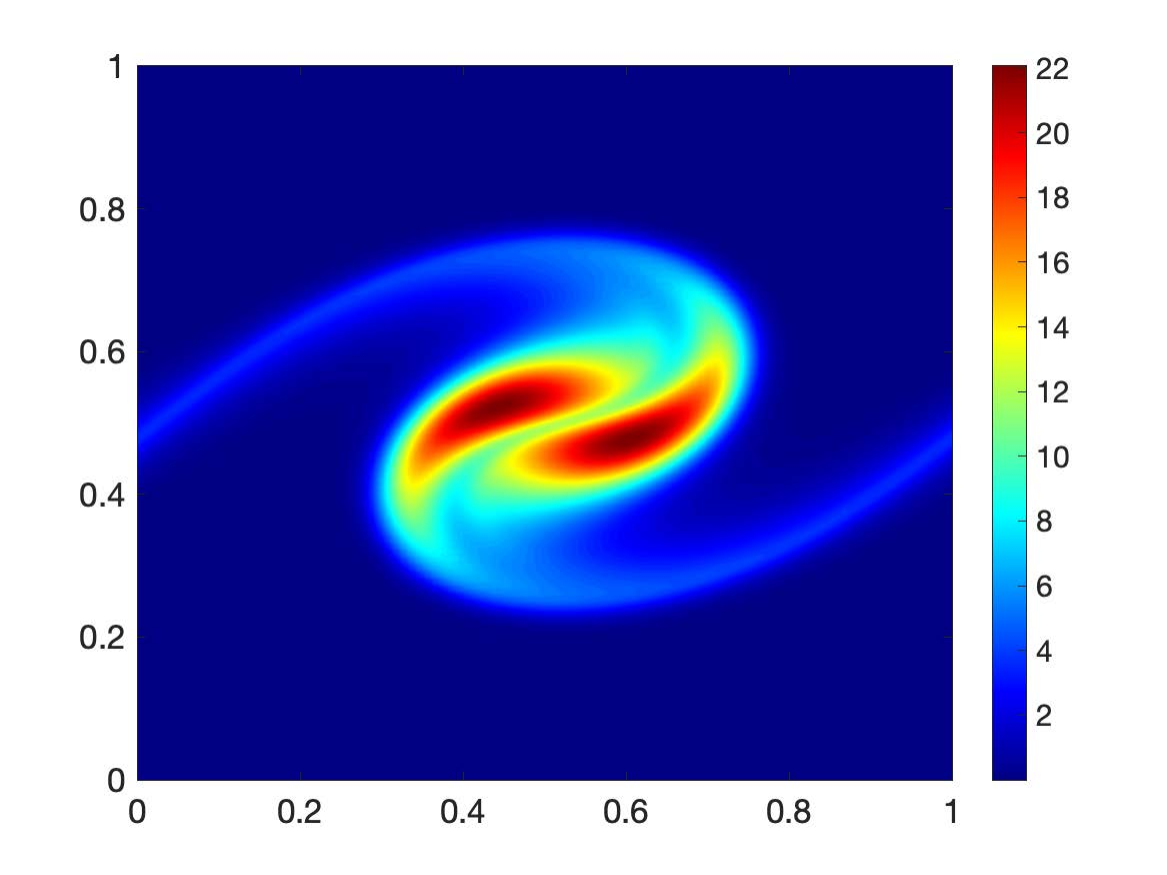}
\includegraphics[width = .24\textwidth, height=.22\textwidth,viewport=65 40 530 400, clip]{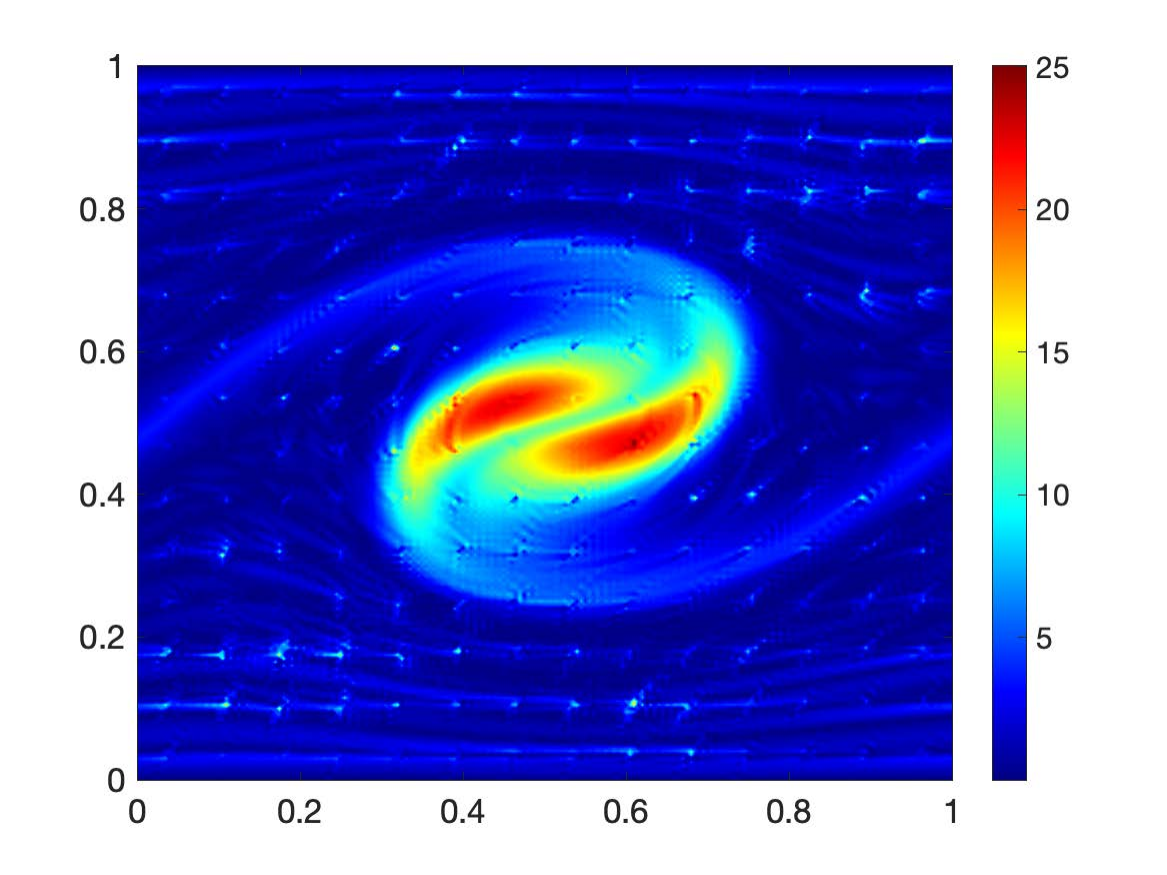}
\includegraphics[width = .24\textwidth, height=.22\textwidth,viewport=65 40 530 400, clip]{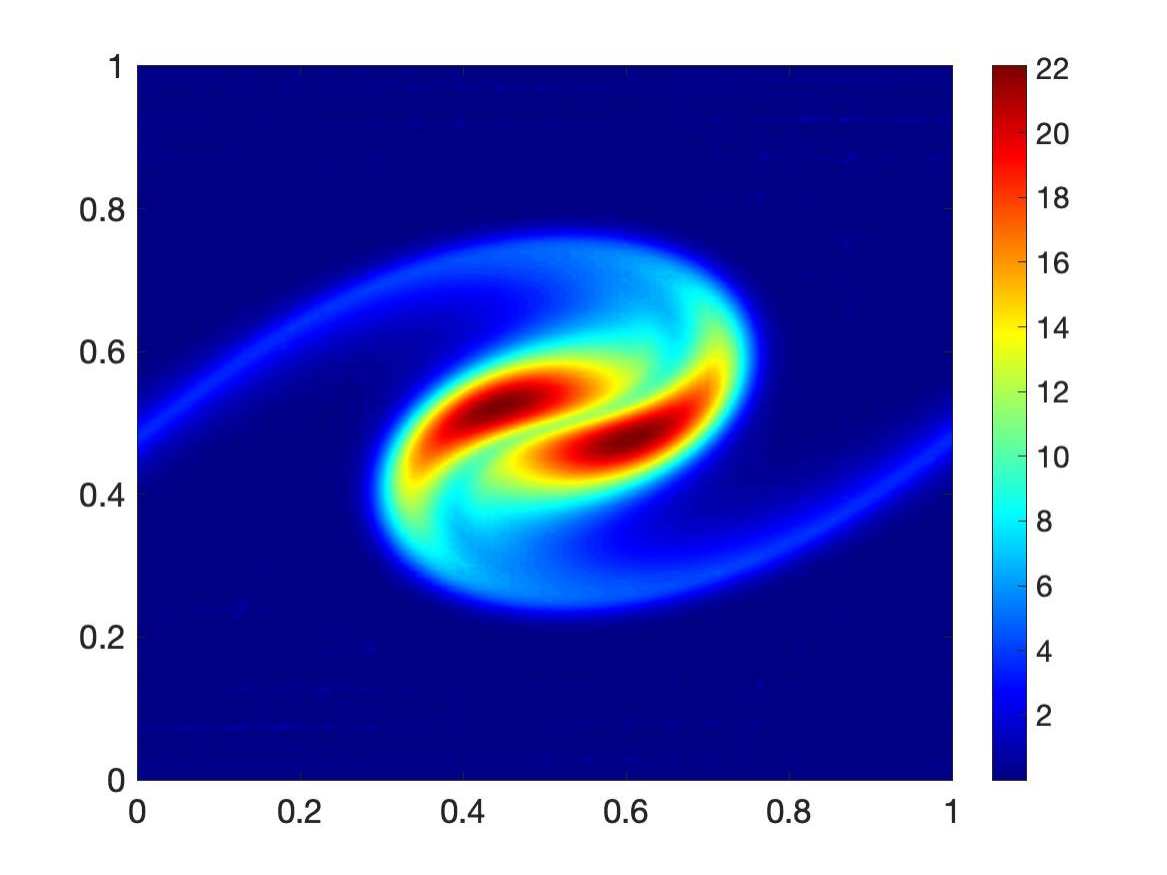}
\includegraphics[width = .24\textwidth, height=.22\textwidth,viewport=65 40 530 400, clip]{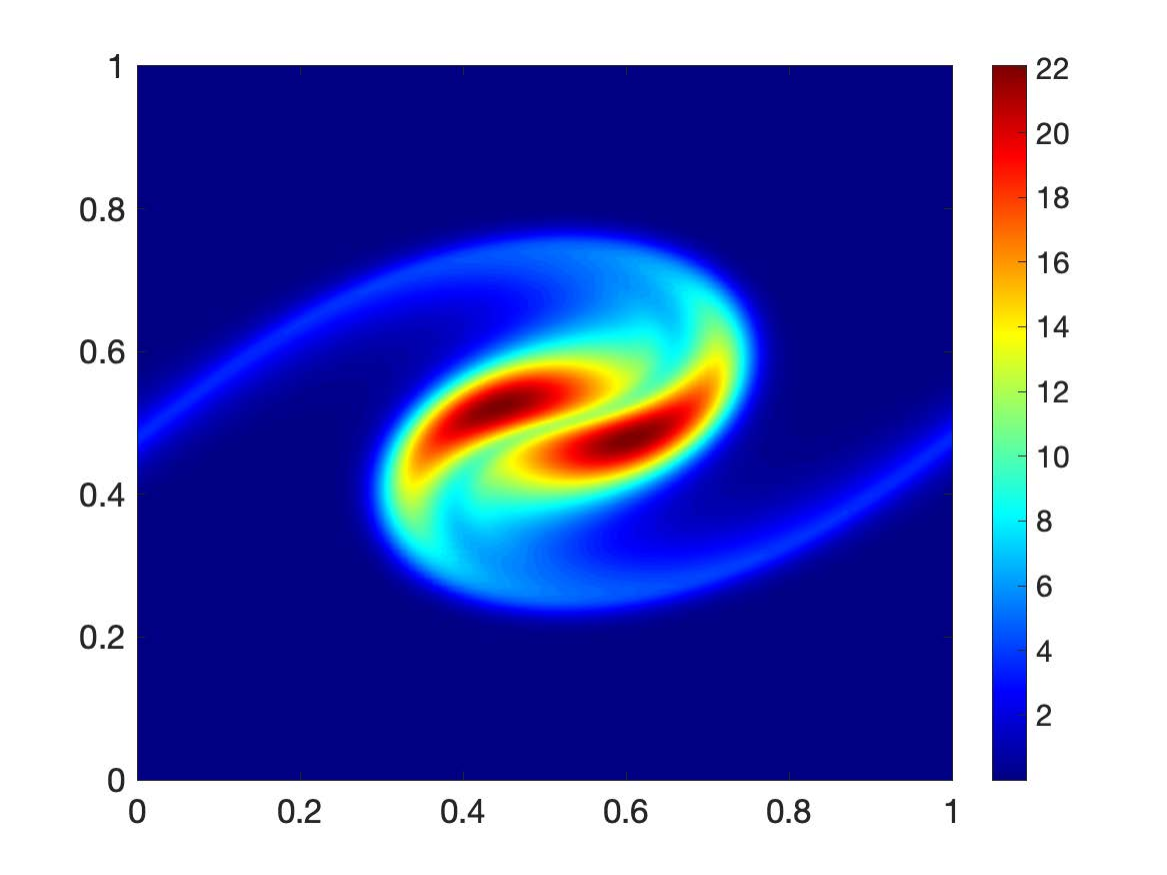}\\
\includegraphics[width = .24\textwidth, height=.22\textwidth,viewport=65 40 530 400, clip]{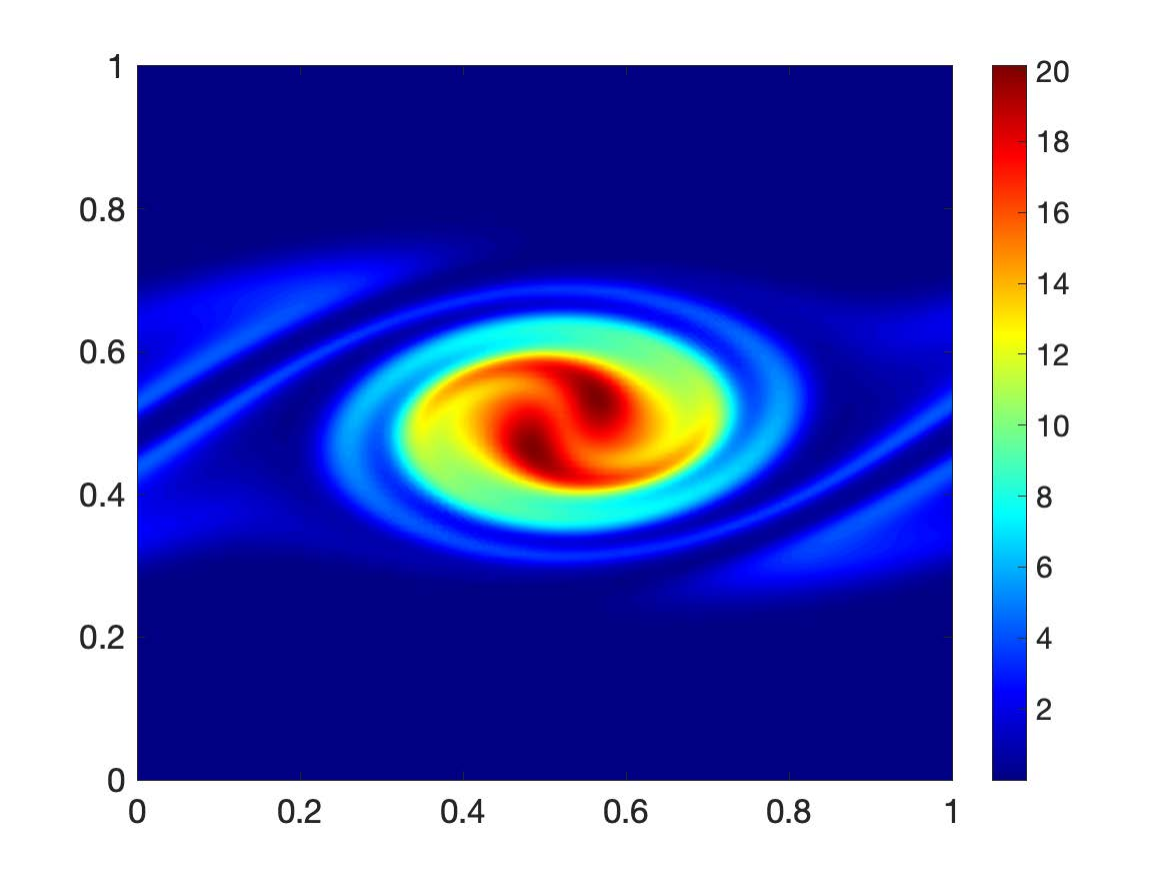}
\includegraphics[width = .24\textwidth, height=.22\textwidth,viewport=65 40 530 400, clip]{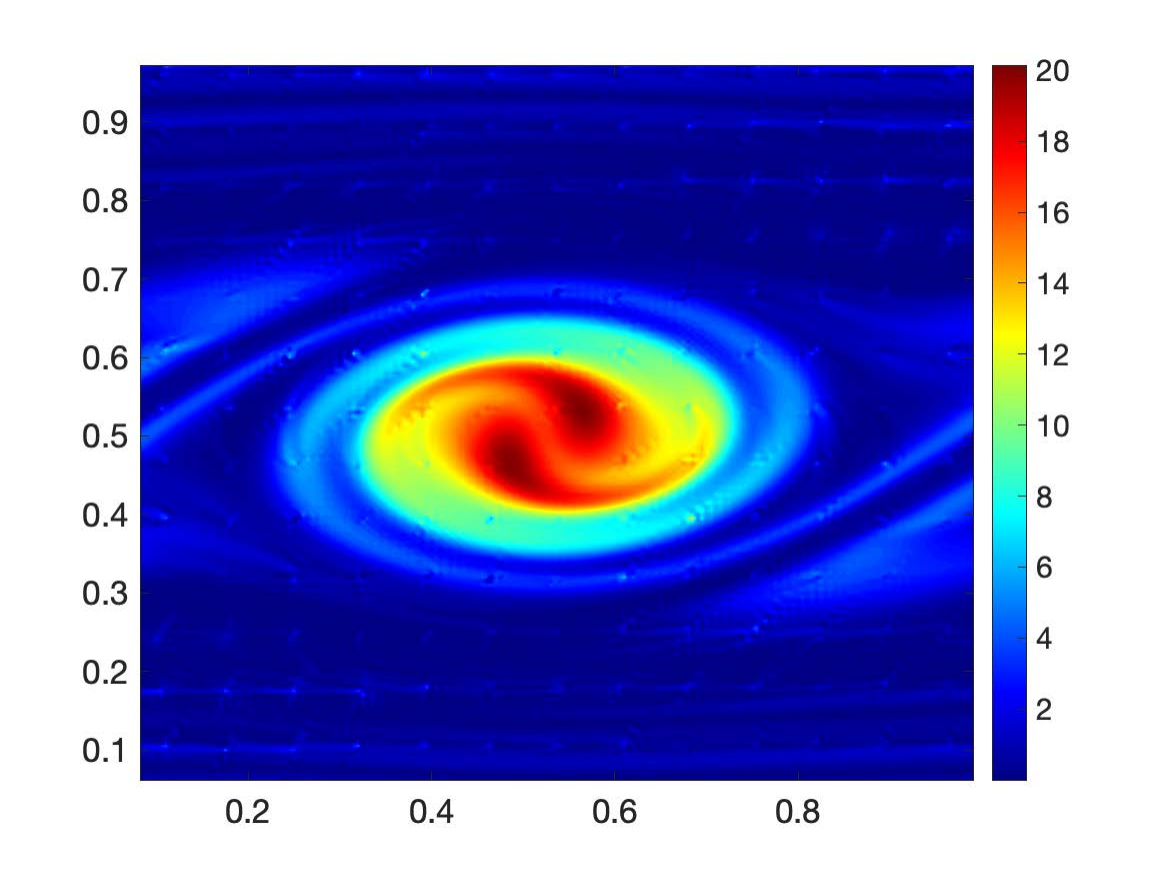}
\includegraphics[width = .24\textwidth, height=.22\textwidth,viewport=65 40 530 400, clip]{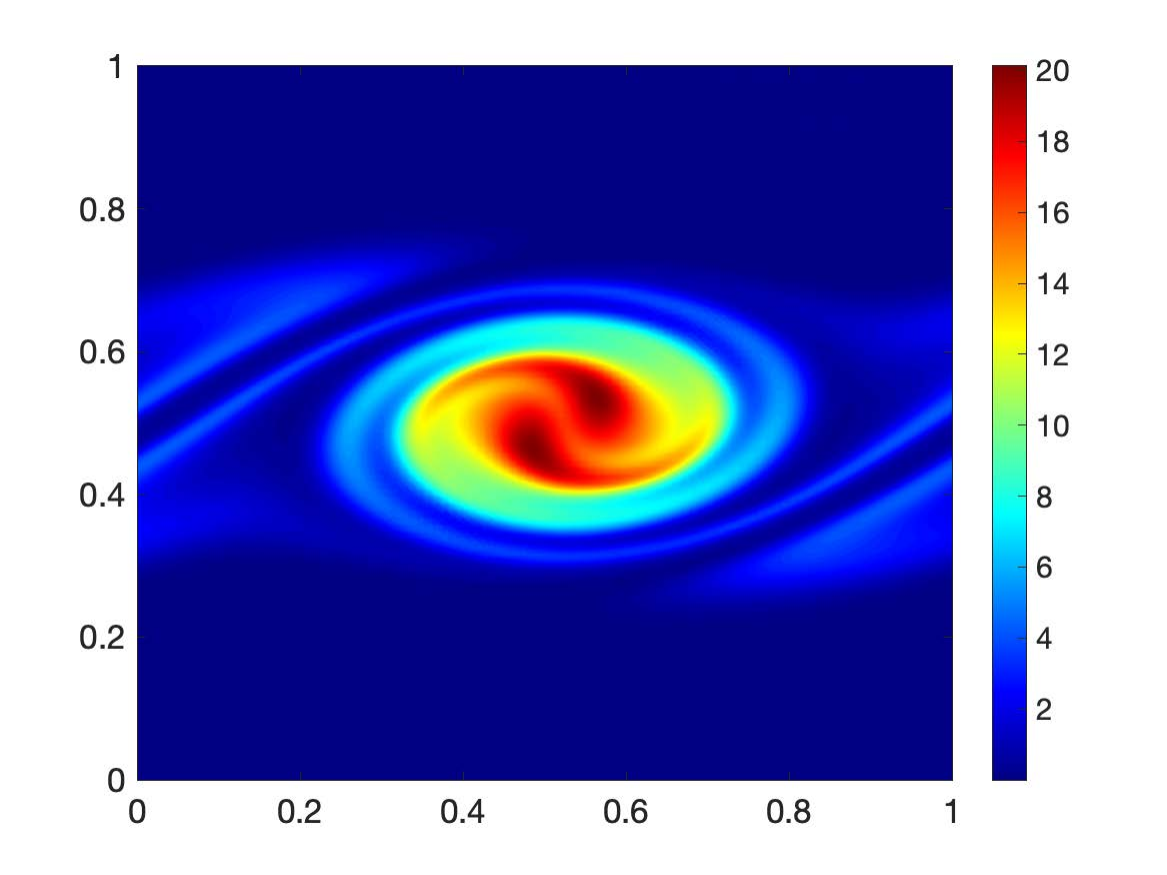}
\includegraphics[width = .24\textwidth, height=.22\textwidth,viewport=65 40 530 400, clip]{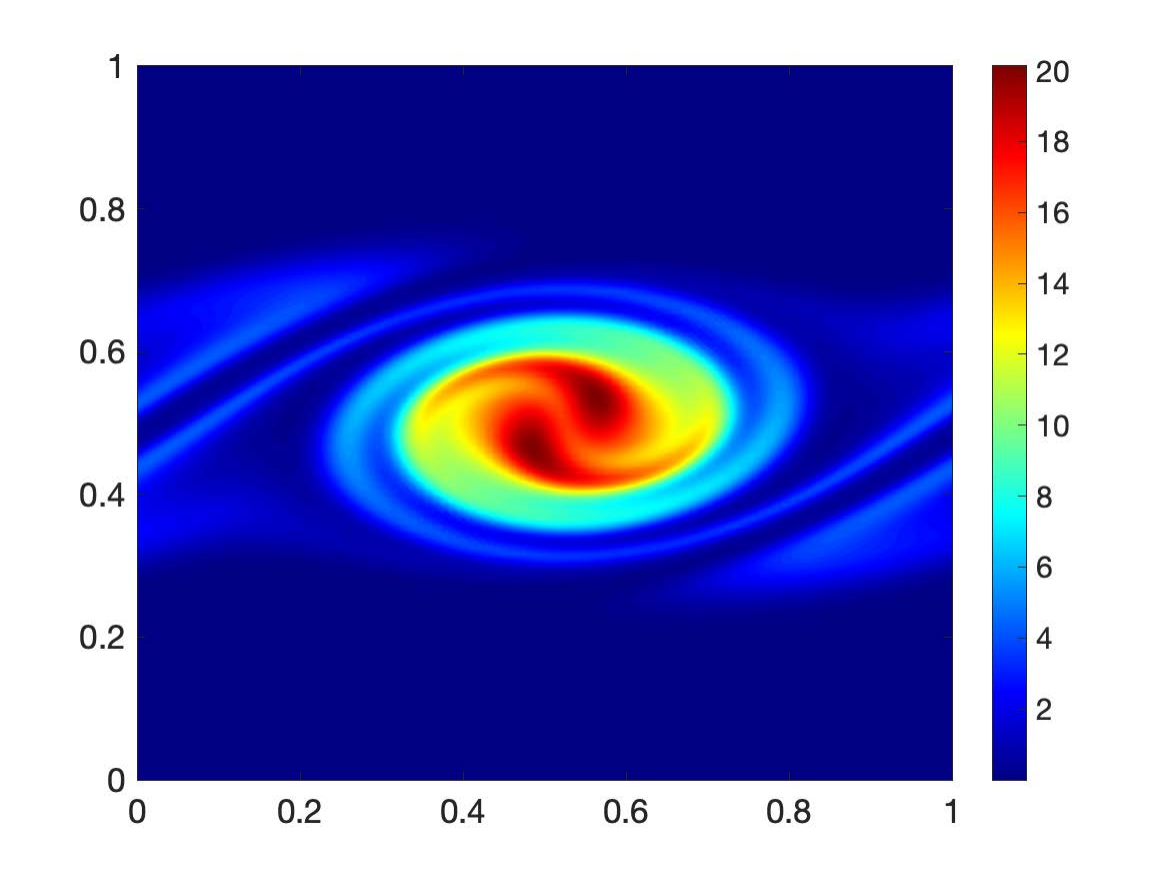}\\
	\caption{\label{KHplots2} Contour plots of DA and DNS absolute vorticity magnitudes at times 2, 5, 8, and 9, reading from top to bottom.}
	\end{center}
\end{figure}

\subsection{The Cahn-Hilliard equation} \label{sec:CH-Num-Exp}

In this subsection, we briefly demonstrate via a numerical experiment that similar CDA projections could be established for different types of evolutionary PDEs. Specifically, we consider the Cahn-Hilliard (CH) equation \cite{Cahn:61:spinodal,CH:58:free,Miranville:19:CH}
\begin{align}\label{eq:continuous-problem}
\begin{split}
\partial_t \varphi - \Delta \left(\varphi^3 - \varphi\right) + \varepsilon^2  \Delta^2 \varphi & = 0, \\
\varphi(\cdot,0) &=\varphi_0,
\end{split}
\end{align}
where $\varphi$ represents the order parameter which takes on values between $-1$ and $1$ and is often interpreted as a concentration of one component in a two component system. The states $\varphi = \pm 1$ indicate phases of pure concentration and $\varepsilon > 0$ is understood to represent an interfacial width between the two phases. The following CDA-FEM algorithm for the CH equation \eqref{eq:continuous-problem} was analyzed in \cite{DR22}: Given $\phi_h^{n-1} \in Z_h$ and true solution $\varphi \in L^\infty(0,T;H_N^2(\Omega))$, find $\phi_h^{n} \in  Z_h$ such that 
\begin{align}\label{eq:fully-discrete-fem}
\left(\frac{\phi_h^{n} - \phi_h^{n-1}}{\Delta t}, \psi\right) + \left(\nabla \left(\left(\phi_h^n\right)^3 - \phi_h^{n-1}\right),\nabla \psi\right) + \varepsilon^2 \aIPh{\phi_h^{n}}{\psi} + \mu \left(I_H \left(\phi_h^n - \varphi^n\right),\psi\right) =0 ,
\end{align}
for all $\psi \in Z_h$ with $Z_h := H^1\left(\Omega\right)\cap P_k(\tau_h(\Omega))$. Here, the bilinear form $\aIPh{\cdot}{\cdot}$ has the following definition:
\begin{align}
\label{eq:aIPh-def}
\aIPh{w}{v} :=& \, \sum_{K \in \tau_h}  \int_K \left(\nabla^2 w : \nabla^2 v\right) \, dx + \sum_{e \in \mathscr{E}_h} \int_e \left\{\mskip-5mu\left\{ \frac{\partial^2 w}{\partial n_e^2} \right\}\mskip-5mu\right\} \left\llbracket \frac{\partial v}{\partial n_e} \right\rrbracket dS 
\nonumber
\\
&+ \sum_{e \in \mathscr{E}_h} \int_e \left\{\mskip-5mu\left\{\frac{\partial^2 v}{\partial n_e^2} \right\}\mskip-5mu\right\}  \left\llbracket \frac{\partial w}{\partial n_e} \right\rrbracket dS + \sigma \sum_{e \in \mathscr{E}_h} \frac{1}{|e|} \int_e  \left\llbracket \frac{\partial w}{\partial n_e} \right\rrbracket  \left\llbracket \frac{\partial v}{\partial n_e} \right\rrbracket dS,
\end{align}
with $\sigma \ge 1$ known as a penalty parameter. The jumps and averages that appear in \eqref{eq:aIPh-def} are defined as follows. For an interior edge $e$ shared by two triangles $K_\pm$ where $n_e$ points from $K_-$ to $K_+$, we define on the edge $e$
\begin{align}
\label{eq:jumps-ave-interior}
\left\llbracket \frac{\partial v}{\partial n_e} \right\rrbracket = n_e \cdot \left(\nabla v_+ - \nabla v_-\right),
\quad
\left\{\mskip-5mu\left\{\frac{\partial^2 v}{\partial n_e^2} \right\}\mskip-5mu\right\} = \frac{1}{2}\left( \frac{\partial^2 v_-}{\partial n_e^2} +  \frac{\partial^2 v_+}{\partial n_e^2} \right),
\end{align}
where $\displaystyle \frac{\partial^2 v}{\partial n_e^2} = n_e \cdot \left(\nabla^2 v\right) n_e$ and $v_\pm = v |_{K_\pm}$. For a boundary edge $e$, we take $n_e$ to be the unit normal pointing towards the outside of $\Omega$ and define on the edge $e$
\begin{align}\label{eq:jumps-ave-boundary}
\left\llbracket \frac{\partial v}{\partial n_e} \right\rrbracket = - n_e \cdot \nabla v_K ,
\quad
\left\{\mskip-5mu\left\{\frac{\partial^2 v}{\partial n_e^2} \right\}\mskip-5mu\right\} = n_e \cdot \left( \nabla^2 v \right) n_e.
\end{align}

In \cite{DR22}, it was once again determined that an upper bound to the nudging parameter $\mu$ was required in the analysis of the FEM CDA scheme. However, the following numerical experiment suggests otherwise, in which we consider a square domain $\Omega = (0,1)^2$ and take $\tau_h(\Omega)$ to be a regular triangulation of $\Omega$ consisting of right isosceles triangles that are a quasi-uniform family. (We use a family of meshes $\tau_h(\Omega)$ such that no triangle in the mesh has more than one edge on the boundary.) 

As the Firedrake Project \cite{R:16:firedrake} was utilized this experiment, we now describe how the data assimilation term $\mu \left(I_H(\phi_h^n - \varphi^n), { I_H } \psi\right)$ was implemented. Specifically, a true solution $\varphi$ was obtained at all times by selecting a cross shaped region as initial conditions, setting the interfacial width parameter $\varepsilon = 0.002$ and the nudging parameter $\mu = 0$, and solving the CH equation using the C$^0$ interior penalty FEM \eqref{eq:fully-discrete-fem}. A data assimilation grid size $H$ was chosen and grid points were identified and located on the finite element mesh. 
%A vector was then created such that the value of 1 was assigned for all nodes corresponding to these grid points and a value of 0 was assigned for all other nodes. Let us name this vector $v$. Then the data assimilation term $\mu \left(I_H(\phi_h^n - \varphi^n),\psi\right)$ was computed by
%\begin{align*}
%\mu \left(I_H(\phi_h^n - \varphi^n),I_H \psi\right) = \mu \left(v \phi_h^n - v \varphi^n, v\psi\right),
%\end{align*}
%where we note that $v \in Z_h$ and that this is equivalent to the interpolation method onto a coarse mesh of piecewise %constants $Z_H$, as described above. 
Finally, the initial conditions for the numerical solution $\phi_0$ was set to random initial conditions.

Figure \ref{fig:grid-error-CH} demonstrates the effectiveness of the CDA-FEM for various grid sizes $H$ with a large nudging parameter of $\mu = 10^9$. Specifically, we choose five different grid sizes $H =$ 0.011049, 0.015625, 0.03125, 0.0625, and 0.125, which correspond respectively to 8,100, 4,096, 1,024, 256 and 64 grid points, while the fine mesh uses piecewise quadratics and has 33,025 grid points.   Theorem 4.3 from \cite{DR22} provides a sufficient condition that the grid size should be chosen as $H = \mathcal{O}( \varepsilon^2)$ but our experiments suggest that a grid size much coarser than that will produce good results if a large enough nudging parameter is chosen. This also suggests that the new theory developed herein should allow for a relaxation of the restriction on the size of the data assimilation grid.

\begin{figure}[h!]
\centering
\includegraphics[width=0.5\textwidth]{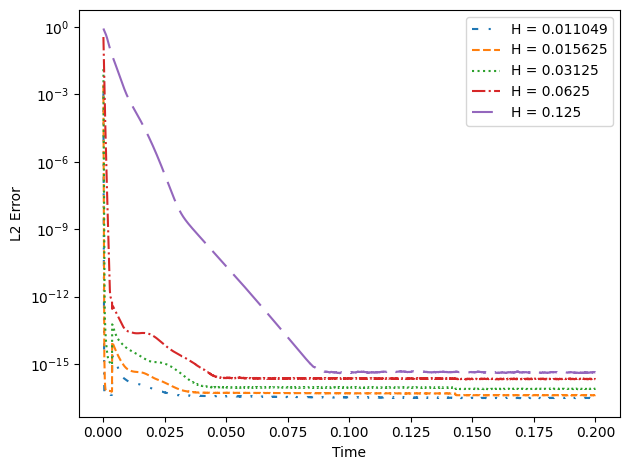}
\caption{The $L^2$ error between the true solution and the solution to the data assimilation finite element method for various data assimilation grid sizes. The mesh size is $h = \nicefrac{\sqrt{2}}{64}$ and the time step size is $\Delta t = \nicefrac{0.002}{32}$. All other parameters are defined in the text. }
\label{fig:grid-error-CH}
\end{figure}

Figure \ref{fig:nudge-error-CH} demonstrates the effectiveness of the CDA-FEM for various values of the nudging parameter $\mu$. For this experiment, we set the data assimilation grid to be $H = 0.03125,$ and choose four different values for the nudging parameter $\mu = 2500, 5000, 10000,$ and $10^9$. Theorem 4.3 from \cite{DR22} admits a sufficient condition that the appropriate value for the nudging parameter $\mu$ is at least $\nicefrac{1}{\varepsilon^2} = 2500$, but if $\mu$ is too large then $H$ needs to be very small.  However, our experiments show that good results can also be obtained for much larger values of $\mu$ without requiring $H$ to be very small.

\begin{figure}[h!]
\centering
\includegraphics[width=0.5\textwidth]{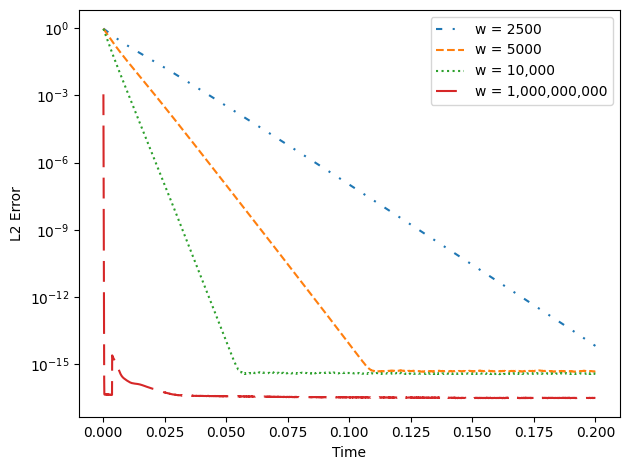}
\caption{The $L^2$ error between the true solution and the solution to the CDA-FEM for various values of the nudging parameter.  The mesh size is $h = \nicefrac{\sqrt{2}}{64}$ and the time step size is $\Delta t = \nicefrac{0.002}{32}$. All other parameters are defined in the text. }
\label{fig:nudge-error-CH}
\end{figure}

\section{Conclusion}\label{Con}
In this paper, we proved optimal long time convergence in $L^2$ for finite element methods with continuous data assimilation without requiring any upper restrictions on the nudging parameter $\mu,$ and without any negative influence of $\mu$ on the error estimate. The goal was achieved by proposing a CDA Poisson projection along with an optimal projection error estimate, in which a harmonic approximation and nodal interpolant played key roles in the the analysis. We then applied the CDA Poisson projection to CDA convergence { analysis} for two typical evolution PDEs: the heat equation and the Navier-Stokes equations. The key contribution here is the fact that with these new projection, we were able to eliminate the scaling of the error with $\mu$ as $\mu$ gets large.   This in turn allows for choosing large or very large $\mu$, with the latter being numerically equivalent to direct Dirichlet enforcement of the measurement data into the discrete solution.%assimilation, %: one only needs to record the locations of observed data and enforce the observations accordingly into the linear system assembled from usual PDE discretization, just as the same way treating Dirichlet boundary condition, 
%which is incredibly simple and computation efficient. 
Multiple numerical tests illustrate the theory.

Future work will consider the case of how to choose $\mu$ for noisy data.  In many applications of CDA, the data measurements will have error present, and it is an open question of how $\mu$ should be chosen if certain properties of the noise are known, such as signal to noise ratio or distribution of the noise.

\section{Acknowledgements}
Authors Xuejian Li and Leo G. Rebholz were partially supported by NSF grant No.~DMS-2152623. Author Amanda E. Diegel was partially supported by NSF grant No.~DMS-2110768. Also, the authors thank Dr. Quyuan Lin for helpful discussions regarding to this work.  



%%\bibliographystyle{elsarticle-num} 


\end{document}